\documentclass[letterpaper,11pt]{amsart}
\usepackage{amsmath}
\usepackage{amssymb}
\usepackage{amscd}
\usepackage{amsthm}
\usepackage{enumerate}
\usepackage{fullpage}
\usepackage{mathrsfs}
\usepackage{setspace}
\usepackage{mathtools}
\usepackage{xcolor}
\usepackage[all]{xy}
\usepackage{url}
\usepackage[colorlinks=true,linkcolor=blue,breaklinks,pagebackref]{hyperref}
\usepackage[lite,abbrev,alphabetic]{amsrefs}
\usepackage[capitalize,nameinlink,noabbrev,nosort]{cleveref}
\hypersetup{
	colorlinks=true,       
	filecolor=blue,      
	urlcolor=blue,           
}

  \newtheoremstyle{mystyle}
    {}
    {}
    {\normalfont}
    {}
    {\bfseries}
    {}
    { }
    {}

\theoremstyle{definition}
\newtheorem{definition}{Definition}[section]

\newtheorem{theorem}[definition]{Theorem}
\newtheorem{proposition}[definition]{Proposition}
\newtheorem{lemma}[definition]{Lemma}
\newtheorem{corollary}[definition]{Corollary}
\newtheorem{conjecture}[definition]{Conjecture}
\newtheorem{hypothesis}[definition]{Hypothesis}

\theoremstyle{mystyle}
\newtheorem*{proposition*}{Proposition}
\newtheorem*{corollary*}{Corollary}
\newtheorem*{conjecture*}{Conjecture}
\newtheorem*{Ac}{Acknowledgement}
\newtheorem*{theorem*}{Theorem}

\theoremstyle{remark}
\newtheorem{remark}[definition]{Remark}

\crefname{definition}{Definition}{Definitions}
\crefname{theorem}{Theorem}{Theorems}
\crefname{proposition}{Proposition}{Propositions}
\crefname{lemma}{Lemma}{Lemmas}
\crefname{corollary}{Corollary}{Corollaries}
\crefname{conjecture}{Conjecture}{Conjectures}
\crefname{hypothesis}{Hypothesis}{Hypotheses}
\crefname{remark}{Remark}{Remarks}
\crefname{condition}{Condition}{Conditions}
\crefname{example}{Example}{Examples}

\def\C{\mathbb{C}}
\def\R{\mathbb{R}}
\def\Q{\mathbb{Q}}
\def\Z{\mathbb{Z}}
\def\H{\mathbb{H}}

\def\GL{\mathrm{GL}}

\def\SL{\mathrm{SL}}

\def\Sp{\mathrm{Sp}}
\def\GSp{\mathrm{GSp}}
\def\M{\mathrm{M}}

\def\rd{\,\mathrm{d}}

\def\idots{\reflectbox{$\ddots$}}

\def\Cc{C_c^{\infty}}

\def\bs{\backslash}

\def\idots{\reflectbox{$\ddots$}}

\DeclareMathOperator{\diag}{diag} 
\DeclareMathOperator{\bc}{bc}
\DeclareMathOperator{\Gal}{Gal}
\DeclareMathOperator{\Res}{Res}
\DeclareMathOperator{\Ind}{Ind}
\DeclareMathOperator{\Hom}{Hom}
\DeclareMathOperator{\rec}{rec}
\DeclareMathOperator{\Cent}{Cent}
\DeclareMathOperator{\Irr}{Irr}
\DeclareMathOperator{\sgn}{sgn}
\DeclareMathOperator{\Nil}{Nil}
\DeclareMathOperator{\tr}{tr}
\DeclareMathOperator{\Nm}{N}

\DeclareMathOperator{\JL}{JL}

\newcommand{\fa}{\mathfrak{a}}

\newcommand{\fg}{\mathfrak{g}}
\newcommand{\fh}{\mathfrak{h}}

\newcommand{\fS}{\mathfrak{S}}

\newcommand{\cC}{\mathcal{C}}

\newcommand{\cL}{\mathcal{L}}

\newcommand{\cN}{\mathcal{N}}
\newcommand{\cO}{\mathcal{O}}

\newcommand{\cT}{\mathcal{T}}

\newcommand{\cW}{\mathcal{W}}

\newcommand{\bfT}{\mathbf{T}}
\newcommand{\bfP}{\mathbf{P}}
\newcommand{\bfN}{\mathbf{N}}
\newcommand{\bfG}{\mathbf{G}}
\newcommand{\bfH}{\mathbf{H}}
\newcommand{\bzero}{\mathbf{0}}

\numberwithin{equation}{section}
\newcommand{\disc}{\mathrm{disc}}
\newcommand{\geom}{\mathrm{geom}}
\newcommand{\spec}{\mathrm{spec}}
\newcommand{\elliptic}{\mathrm{ell}}

\newcommand{\semi}{\mathrm{ss}}
\newcommand{\reg}{\mathrm{reg}}
\newcommand{\gen}{\mathrm{gen}}

\DefineSimpleKey{bib}{arxiv}
\newcommand\myurl[1]{\url{#1}}
\newcommand\arxiv[1]{available at \href{https://arxiv.org/abs/#1}{arXiv:#1}}
\def\MR#1{\quad \href{http://www.ams.org/mathscinet-getitem?mr=#1}{MR#1}}
\BibSpec{article}{%
  +{}  {\PrintAuthors}                {author}
  +{,\,\,} {\textit}                     {title}
  +{.} { }                            {part}
  +{:} {\textit}                     {subtitle}
  +{,} {\PrintContributions}         {contribution}
  +{,} {\PrintConference}            {conference}
  +{}  {\PrintBook}                   {book}
  +{,} { }                            {booktitle}
  +{,} { }                            {journal}
  +{,\,\,} {\textbf}                     {volume}
  +{,\,\,} {\issuetext}                     {number}
  +{,\,\,} {\PrintDateB }                 {date}
  +{,\,\,} {pp.}                        {pages}
  +{,\,\,} {available at \eprint}        {eprint}
  +{;} { \arxiv}               {arxiv}
  +{.\!\!}  {\SentenceSpace \PrintReviews} {review}
  +{,\,\,} {\PrintDOI}                   {doi}
}

\title{A reformulation of the conjecture of Prasad and Takloo-Bighash}

\author{Miyu Suzuki\vspace{-7ex}}

\address{Miyu Suzuki \\
Department of Mathematics Faculty of Science\\
Kyoto University\\
Kitashirakawa Oiwake-cho, Sakyo-ku,
Kyoto 606-8502, Japan}
\email{suzuki.miyu.4c@kyoto-u.ac.jp}

\begin{document}

\maketitle
\medskip

\begin{center}
\emph{Dedicated to Tamotsu Ikeda on the occasion of his sixtieth birthday}
\end{center}

\begin{abstract}
Prasad and Takloo-Bighash proposed a conjecture which predicts a necessary condition in terms of epsilon factors for representations of $\GL_n(F)$ and its inner forms to have linear periods.
In this rather expository article,  we reformulate their conjecture in the following form: The distinguished members in each generic $L$-packet $\Pi_\phi$ are determined by the characters of the component group $S_\phi$ and local epsilon factors.
We follow Aubert et al.\,for the definitions of the $L$-packets and the component groups.

We observe that under some hypotheses,  the reformulated conjecture follows from the conjectural multiplicity formula recently proposed by Chen Wan for general spherical varieties and the conjectural integral formula for epsilon factors which we propose in this article.
\end{abstract}

\section{Introduction}
\label{sec:introduction}
   

Let $F$ be a local field of characteristic $0$.
Let $G$ be a reductive algebraic group over $F$ and $H$ a closed subgroup of $G$.
Suppose that a Borel aubgroup of $G$ has an open orbit in the homogeneous space $G/H$.
Such a pair $(G,  H)$ is called a \emph{spherical pair}. 

By a \emph{representation} of $G(F)$,  we mean a smooth admissible representation when $F$ is non-archimedean and a smooth admissible moderate growth Fr\'echet representation when $F$ is archimedean.
Let $\Irr(G(F))$ denote the set of isomorphism classes of irreducible representations of $G(F)$. 

Let $\chi$ be a character of $H(F)$.
For $\pi\in\Irr(G(F))$,  $m(\pi,  \chi):=\dim\Hom_{H(F)}(\pi,  \chi)$ is called the multiplicity of $(H,  \chi)$-periods on $\pi$.
We say that $\pi$ is \emph{$(H,  \chi)$-distinguished} if $m(\pi,  \chi)\neq0$.

To simplify the situation,  we focus on the case where $G$ is a classical group (or a product of classical groups).
In the local Langlands conjecture,  it is expected that there is a bijection between a Vogan $L$-packet $\Pi_\phi$ and the set of irreducible representations of the component group $S_\phi$.
See \cite[Section 9]{GGP12} for details.
For several spherical pairs $(G,  H)$,  there are conjectures which assert that each generic $L$-packet $\Pi_\phi$ of $G(F)$ contains unique distinguished member.
Moreover,  that unique distinguished representation corresponds to a character of $S_\phi$,  which is defined using epsilon factors.
Such a conjecture is called an \emph{epsilon dichotomy conjecture}.

The most typical example is the Gan-Gross-Prasad conjecture \cite[Conjecture 17.3]{GGP12}.
This conjecture is proved by Waldspurger and M\oe glin in the case of $p$-adic special orthogonal groups.

The key ingredient of the proof is the multiplicity formula $m(\pi,  \chi)=m_\geom(\pi,  \chi)$.
Here,  $m_\geom(\pi,  \chi)$ is called the geometric multiplicity,  which is given by some integrals involving the germs of the character $\theta_\pi$ of $\pi$.
In \cite{Wal10} and \cite{Wal12a},  Waldspurger proved the geometric and spectral expansions of a local relative trace formula and deduced the multiplicity formula for tempered representations.

In \cite{Wal12b},  he also proved the geometric and spectral expansions of a twisted local relative trace formula and obtained an integral formula for the epsilon factor,  which has a similar form to the multiplicity formula.
Combining these integral formulae with the endoscopic character relation,  he proved the Gan-Gross-Prasad conjecture for tempered representations of the $p$-adic special orthogonal groups in \cite{Wal12c}.
Finally,  M\oe glin and Waldspurger \cite{MW} reduced the problem for generic representations to the case of tempered ones and completed the proof.

This method is applied to many other situations,  for example by Beuzart-Plessis \cite{BP14},  \cite{BP15},  \cite{BP16},  \cite{BP18},  Wan \cite{Wan19a},  \cite{Wan19b},  Beuzart-Plessis and Wan \cite{BPW19},  \cite{BPW},  Wan and Zhang \cite{WZa},  \cite{WZb},  \cite{WZc}.
The archimedean cases of the Gan-Gross-Prasad conjecture is studied along the same line by Beuzart-Plessis \cite{BP20},   Luo \cite{Luo},  Chen \cite{Chen} and Chen and Luo \cite{CL}.
It is also remarkable that Wan \cite{Wan22} found a uniform definition of the geometric multiplicity $m_\geom(\pi,  \chi)$ and proposed a conjectural multiplicity formula for general spherical varieties.

\subsection{The original conjecture}
In \cite{PTB11},  Prasad and Takloo-Bighash proposed an epsilon dichotomy conjecture for linear periods of representations of $\GL_n(F)$ and its inner forms.

Let $F$ be a finite extension of $\Q_p$ or $F=\R$.
Take a positive integer $m$ and a central division algebra $D$ over $F$ of dimension $\dim_F(D)=d^2$.
We put $G=\GL_m(D)$ and $n=md$.
Then $G$ is an inner form of $\GL_n(F)$.
Let $Z_G$ denote the center of $G$.
For $\pi\in\Irr(G(F))$,  let $\omega_\pi$ be the central character.
Let $\Irr_\disc(G(F))$ denote the set of irreducible essentially square integrable representations of $G(F)$.

Let $W_F$ be the Weil group of $F$ and $WD_F$ the Weil-Deligne group of $F$,  which is
    \[
    WD_F=
        \begin{cases}
        W_F\times\SL_2(\C) & \text{if $F$ is a finite extension of $\Q_p$},  \\
        W_F & \text{if $F=\R$}.
        \end{cases}
    \]
By the local Langlands correspondence for $G(F)$,  each irreducible representation $\pi\in\Irr(G(F))$ corresponds to an $L$-parameter $\phi_\pi\,\colon WD_F\rightarrow\GL_n(\C)$.

Let $E$ be a quadratic extension of $F$ and $\chi$ a character of $E^\times$.
Let $\eta_{E/F}$ be the quadratic character of $F^\times$ corresponding to the extension $E/F$ via the local class field theory.

Assume that $n$ is even so that there is an embedding $E\hookrightarrow \M_m(D)$.
Such an embedding is unique up to conjugation by the Skolem-Noether theorem.
Let $H$ be the centralizer of $E^\times$ in $G$.
The group $H$ is an inner form of $\GL_{n/2}(E)$ hence is isomorphic to $\GL_r(C)$,  where $C$ is a central division algebra over $E$ such that
    \[
    C \cong 
        \begin{cases}
        \Cent_D(E) & \text{if $d$ is even} \\
        D\otimes_F E & \text{if $d$ is odd},
        \end{cases} \qquad 
    r = 
        \begin{cases}
        m & \text{if $d$ is even} \\
        \frac{m}{2} & \text{if $d$ is odd}.
        \end{cases}
    \]
Note that if $d$ is even,  $E$ is embedded into $D$ and $\Cent_D(E)$ denotes the centralizer of $E$ in $D$.

We write by $\chi_H$ the character of $H(F)$ obtained by composing $\chi$ with the reduced norm of the central simple algebra $\M_r(C)$ over $E$.
We say that a representation $\pi$ of $G(F)$ has \emph{linear periods} with respect to $(H,  \chi_H)$ if $m(\pi,  \chi)\neq0$.

\begin{conjecture}\label{conj:PTB}
Let $\pi\in\Irr(G(F))$ and $\phi\,\colon WD_F\rightarrow\GL_n(\C)$ be its $L$-parameter.
Let $\chi$ be a character on $E^\times$.
If $m(\pi,  \chi)\neq0$,  then
\setlength{\leftmargini}{25pt}
\begin{itemize}
\item[(1)]
the $L$-parameter $\phi$ takes values in $\GSp_n(\C)$ with similitude factor $\chi|_{F^\times}$.
Here,  $\chi|_{F^\times}$ is regarded as a character on $W_F$ via the reciprocity isomorphism $W_F^{ab}\cong F^\times$;
\item[(2)]
the root number satisfies 
$\varepsilon\left(\phi\otimes\Ind_{W_E}^{W_F}(\chi^{-1})\right)=(-1)^m\chi\eta_{E/F}(-1)^{\frac{n}{2}}$. 
\end{itemize}
For $\pi\in\Irr_\disc(G(F))$,  the converse holds, \emph{i.e.} if it satisfies (1) and (2),  then $m(\pi,  \chi)\neq0$.
\end{conjecture}

Note that the last part of \cref{conj:PTB} is rephrased as follows.
Let $\pi\in\Irr_\disc(G(F))$,  $\phi$ be its $L$-parameter and set $\bc_{E/F}(\phi)=\phi|_{WD_E}$.
Then 
    \begin{equation}\label{eq:0}
    m(\pi,  \chi)=\frac12\left\{1+(-1)^m\varepsilon\left(\bc_{E/F}(\phi)\otimes\chi^{-1}\right)\chi(-1)^{\frac{n}{2}}\right\}
    \end{equation}
if $\phi$ satisfies (1) of \cref{conj:PTB} and $m(\pi,  \chi)=0$ otherwise.

The original conjecture \cite[Conjecture 1]{PTB11} is stated only when $F$ is non-archimedean and $\pi$ is ``generic'' but these restrictions are known to be unnecessary.
There has been recent progress toward this conjecture.

\begin{theorem}\label{thm:known_cases}
\cref{conj:PTB} holds in the following cases:
\setlength{\leftmargini}{25pt}
\begin{itemize}
\item[(a)] when $F$ is a finite extension of $\Q_p$ with $p\neq 2$ and $\chi=\mathbf{1}$ is the trivial character,
\item[(b)] when $F=\R$.
\end{itemize}
\end{theorem}

\begin{proof}
The non-archimedean case (a) is a combination of several results from \cite{Sec},  \cite{Suz21},  \cite{Xue21} and \cite{SX23}.
The archimedean case (b) is \cite[Theorem 1.1]{ST23}.
\end{proof}

The main result of this article is that the equation \eqref{eq:0} follows from the conjectural multiplicity formula (\cref{conj:mult_formula}) and the  conjectural expression for the root number (\cref{conj:epsilon}).

Let $\bfG=\Res_{E/F}\GL_n$,  where $\Res_{E/F}$ denotes the Weil restriction of scalars.
Take $\pi\in\Irr_\disc(G(F))$.
WE write the Jacquet-Langlands transfer of $\pi$ to $\GL_n(F)$ by $\pi_0$.
Let $\pi'=\bc_{E/F}(\pi_0)\in\Irr(\bfG(F))$ be the base change of $\pi_0$ to $\bfG(F)$.

\begin{theorem*}\textbf{\ref{thm:epsilon}.}
In the above situation,  we set $\varepsilon(\pi',  \chi)=\varepsilon(\frac12,  \pi'\otimes\chi^{-1},  \psi_E)\chi(-1)^{\frac{n}{2}}$,  where $\psi_E$ is a non-trivial additive character of $E$.
Assume that 
\begin{itemize}
\item \cref{hyp:2},  \cref{conj:mult_formula} and \cref{conj:epsilon} hold.
\item $\pi'\in\Irr_\disc(\bfG(F))$.
\end{itemize}
Then we have 
    \[
    m(\pi,  \chi)=\frac12(1+\chi_\pi(-1)\varepsilon(\pi',  \chi))
    \]
if $\rec_{\bfG(F)}(\pi')$ takes values in $\GSp_n(\C)$ with similitude factor $\chi\circ\Nm_{E/F}$ and  $m(\pi,  \chi)=0$ otherwise.
Here,  $\chi_\pi$ is the character of the component group $S_\phi$ attached to $\pi$.
\end{theorem*}

\subsection{A reformulation}
\cref{conj:PTB} predicts a necessary condition for distinction in terms of epsilon factors.
It is natural to expect that we can rewrite their conjecture in a form similar to that of the Gan-Gross-Prasad conjecture,  that is,  the distinguished members in each generic $L$-packet $\Pi_\phi$ is determined by using a character of the component group $S_\phi$ and  epsilon factors.

The obvious obstacle to do so is,  in the usual formulation of the local Langlands correspondence for $\GL_n(F)$ and its inner forms,  every $L$-packet $\Pi_\phi$ is a singleton and the component group $S_\phi$ is the trivial group,  hence it seems meaningless to consider characters of $S_\phi$.
Following \cite[Section 2]{ABPS16},  we can treat the representations of $\GL_n(F)$ and its inner forms simultaneously and obtain $L$-packets which are not singletons.

In the rest of this section,  suppose that $F$ is non-archimedean.
For an $L$-parameter $\phi\,\colon WD_F\to\GL_n(\C)$,  let $\Pi_\phi$ be the set of representations of all inner forms of $\GL_n(F)$ corresponding to $\phi$.
we say that $\Pi_\phi$ is generic if it contains a generic representation of $\GL_n(F)$.
We also define the component group $S_\phi$ as
    \[
    S_\phi=Z_{\SL_n(\C)}(\phi)/Z_{\SL_n(\C)}(\phi)^\circ,
    \]
where $Z_{\SL_n(\C)}(\phi)$ is the centralizer of the image of $\phi$ by $\SL_n(\C)$ and $Z_{\SL_n(\C)}(\phi)^\circ$ is its identity component.
There is a natural bijection $\pi\leftrightarrow\chi_\pi$ between $\Pi_\phi$ and the set of characters of $S_\phi$.
Using these $L$-packet $\Pi_\phi$ and the component group $S_\phi$,  we can reformulate \cref{conj:PTB} as follows.

\begin{conjecture*}\textbf{\ref{conj:reform}.}
Suppose that $\phi$ takes values in $\GSp_n(\C)$ with similitude factor $\chi|_{F^\times}$ and $\Pi_\phi$ is generic.
We write $\phi$ as  
    \[
    \phi=\bigoplus_{i\in I_\phi}\phi_i \oplus 
    \bigoplus_{j\in J_\phi}(\phi_j \oplus (\phi_j^\vee\cdot\chi|_{F^\times})),  
    \]
where $I_\phi,  J_\phi$ are finite sets and each $\phi_k\,\colon WD_F\to\GL_{n_k}(\C)$ is a discrete $L$-parameter which satisfies
\begin{itemize}
\item for each $i\in I_\phi$,  $n_i$ is even and $\phi_i$ takes values in $\GSp_{n_i}(\C)$ with similitude factor $\chi|_{F^\times}$,
\item  for any $i,  i'\in I_\phi$,  we have $\phi_i\not\cong\phi_{i'}$ if $i\neq i'$.
\end{itemize}
Then for $\pi\in\Pi_\phi$,  $m(\pi,  \chi)\neq0$ if and only if 
    \[
    \varepsilon\left(\phi_i\otimes\Ind_{W_E}^{W_F}(\chi^{-1})\right)
    =\left(\chi_\pi(\zeta_n)\chi\eta_{E/F}(-1)\right)^{\frac{n_i}{2}},  \qquad i\in I_\phi,
    \]
where $\zeta_n:=\exp(2\pi\sqrt{-1}/n)$ is a primitive $n$-th root of unity.
\end{conjecture*}

\subsection{The conjectures on multiplicity formula}

For $\pi\in\Irr(G(F))$,  the geometric multiplicity $m_\geom(\pi,  \chi)$ is defined in \cite[Definition 7.1]{Wan22}.
See \cref{def:geom_mult1} and \cref{def:geom_mult2} for the precise definition.

\begin{conjecture*}\textbf{\ref{conj:mult_formula}.}
For $\pi\in\Irr_\disc(G(F))$,  we have $m(\pi,  \chi)=m_\geom(\pi,  \chi)$.
\end{conjecture*}

Under some hypothesis (\cref{hyp:2}) on the Shalika period,  \cref{conj:mult_formula} provides a simple formula for the number of distinguished representations in $ \Pi_\phi$ when $\phi$ is a discrete $L$-parameter.
\begin{proposition*}\textbf{\ref{prop:deduce}.}
Assume \cref{hyp:2} and \cref{conj:mult_formula} hold.
For $\phi\in\Phi_\disc(\GL_n(F))$,  
    \[
     \sum_{\pi\in\Pi_\phi}m(\pi,  \chi) =
        \begin{cases}
        n/2 & \text{if $\phi$ takes values in $\GSp_n(\C)$ with similitude factor $\chi|_{F^\times}$},  \\
        0 & \text{otherwise}.
        \end{cases}
    \]
\end{proposition*}

In particular,  we see that \cref{conj:reform} for $\chi=\mathbf{1}$ follows from the multiplicity formula.

\begin{corollary*}\textbf{\ref{cor:trivial}.}
When $\chi=\mathbf{1}$,  \cref{conj:mult_formula} implies \cref{conj:reform}.
\end{corollary*}

Suppose that $\pi\in\Irr_\disc(G(F))$ and $\pi_0\in\Irr_\disc(\GL_n(F))$ are in $\Pi_\phi$.
Let $\pi'\in\Irr(\GL_n(E))$ which corresponds to $\bc_{E/F}(\phi):=\phi|_{WD_E}$.
Assume that $\pi'\in\Irr_\disc(GL_n(E))$. 
In \cref{sec:root_number},  we define $\varepsilon_\geom(\pi',  \chi)$ which is given by integrals involving the germs of the twisted character of $\pi'$.

\begin{conjecture*}\textbf{\ref{conj:epsilon}.}\label{conj:2}
We have $\varepsilon_\geom(\pi',  \chi)=\varepsilon\left(\bc_{E/F}(\phi)\otimes\chi^{-1}\right)\chi(-1)^{\frac{n}{2}}$. 
\end{conjecture*}

Our main result is a consequence of \cref{conj:epsilon} combined with \cref{conj:mult_formula}.
The relations among these conjectures and results together with other ones in this article are summarized as follows.
    \[
    \xymatrix @R=15pt @C=50pt{
    &
    \setlength{\fboxsep}{0pt}
    \fbox{\begin{minipage}{63pt}
    \begin{tabular}{c}
     Conjecture  \\
    \ref{conj:geom_exp}
    \end{tabular}
    \end{minipage}}\ar @{=>}^{+\text{\cref{prop:geom_exp}}}[d]
    & \\ 
    \setlength{\fboxsep}{0pt}
    \fbox{\begin{minipage}{63pt}
    \begin{tabular}{c}
     Conjecture  \\
    \ref{conj:geom_exp2}
    \end{tabular}
    \end{minipage}}\ar @{=>}_{\text{\cref{hyp:2}}+}
    ^{\substack{+\text{\cref{hyp:3}} \\ +\text{\cref{hyp:4}}}}[d]
    \hspace{-2.5em} 
    &
    \setlength{\fboxsep}{0pt}
    \fbox{\begin{minipage}{63pt}
    \begin{tabular}{c}
     Conjecture  \\
    \ref{conj:mult_formula}
    \end{tabular}
    \end{minipage}}\ar @{=>}^{+\text{\cref{hyp:2}}}[d] \ar @{=>}^(0.47){\text{\cref{cor:trivial}}}[r]
    &
    \setlength{\fboxsep}{0pt}
    \fbox{\begin{minipage}{80pt}
    \begin{tabular}{c}
    \cref{conj:reform} \\
    for $\chi=\mathbf{1}$
    \end{tabular}
    \end{minipage}}                             \\
    \setlength{\fboxsep}{0pt}
    \fbox{\begin{minipage}{63pt}
    \begin{tabular}{c}
    Conjecture \\
    \ref{conj:epsilon}
    \end{tabular}
    \end{minipage}} \quad +\hspace{-4em} 
    &  
    \setlength{\fboxsep}{0pt}
    \fbox{\begin{minipage}{65pt}
    \begin{tabular}{c}
    Prospoition  \\
    \ref{prop:deduce}
    \end{tabular}
    \end{minipage}}  \ar @{=>}[r]
    & 
    \setlength{\fboxsep}{0pt}
    \fbox{\begin{minipage}{55pt}
    \begin{tabular}{c}
    Theorem  \\
    \ref{thm:epsilon}
    \end{tabular}
    \end{minipage}} 
    }\]

\begin{remark}
The anonymous referee suggested the problem extending \cref{conj:reform} to non-generic $L$-packets and relating that generalization with a general multiplicity formula.
It is natural to ask whether we can formulate such a generalized epsilon dichotomy conjecture along the lines of the non-tempered Gan-Gross-Prasad conjecture \cite{GGP20}.
This question in relation with the conjectural multiplicity formula \cite[Conjecture 7.6]{Wan 22} for general representations should be our future work.
\end{remark}

\section{Preliminaries}
\label{sec:preliminary}

From now on,  suppose that $F$ is a finite extension of $\Q_p$.
The archimedean case is discussed later in \cref{rem:archimedean}.
Let $h(D)$ denote the Hasse invariant of $D$,  which is a primitive $d$-th root of unity.

\subsection{Langlands correspondence for $G(F)$}
\label{sec:Langlands}

This subsection is a summary of the facts about the local Langlands correspondence and the local Jacquet-Langlands correspondence for $G(F)=\GL_m(D)$.
For details,  see \cite[Section 2]{ABPS16},  and \cite{DKV84}.
We may remove the condition that $n=md$ is even.

We write the set of regular semisimple elements in $G(F)$ (resp.\,$\GL_n(F)$) by $G_\reg(F)$ (resp.\,$\GL_{n, \reg}(F)$).
For $g\in G_\reg(F)$ and $g'\in\GL_{n,  \reg}(F)$,  we write $g\leftrightarrow g'$ if they have the same characteristic polynomial.
For a representation $\pi$ of $G(F)$ or $\GL_n(F)$,  let $\theta_\pi$ denote its character.

\begin{theorem}[\cite{DKV84}]\label{thm:DKV}
There is a unique bijection $\JL\,\colon \Irr_\disc(G(F))\to\Irr_\disc(\GL_n(F))$ such that for $\pi\in\Irr_\disc(G(F))$,  we have
    \[
    (-1)^m\theta_\pi(g)=(-1)^n\theta_{\JL(\pi)}(g')
    \]
for all $g\in G_\reg(F)$ and $g'\in\GL_{n,  \reg}(F)$ such that $g\leftrightarrow g'$.
The map $\JL$ is called the local Jacquet-Langlands correspondence.
\end{theorem}

Let $\Phi(G(F))$ be the set of equivalence classes of $L$-parameters $\phi\,\colon WD_F\to\GL_n(\C)$ which is relevant to $G$.
Note that $\Phi(G(F))\subsetneq\Phi(\GL_n(F))$ if $G(F)$ is not split.
An $L$-parameter $\phi\in\Phi(G(F))$ is called \emph{discrete} if it is irreducible seen as a representation of $WD_F$ on $\C^n$.
This is equivalent to say that the centralizer of its image in $\GL_n(\C)$ is isomorphic to $\C^\times$.
Let $\Phi_\disc(G(F))$ denote the subset of discrete $L$-parameters.

The local Langlands correspondence for $\GL_n(F)$ established by \cite{HT01} and \cite{Hen00} combined with \cref{thm:DKV} provides a canonical injection
    \[
    \rec_{G(F)}\,\colon \Irr(G(F))\to \Phi(\GL_n(F))
    \]
whose image is $\Phi(G(F))$.
Note that for each $\phi\in\Phi(\GL_n(F))$,  the fiber $\rec_{G(F)}^{-1}(\phi)$ is a singleton if $\phi\in\Phi(G(F))$ and empty otherwise.
The map $\rec_{G(F)}$ has the following properties:
\begin{itemize}
\item For $\pi\in\Irr(G(F))$,  we have $\rec_{G(F)}(\pi^\vee)=\rec_{G(F)}(\pi)^\vee$,  where $^\vee$ denotes the contragredient.
\item The restriction of $\rec_{G(F)}$ to $\Irr_\disc(G(F))$ induces a bijection $\Irr_\disc(G(F))\to\Phi_\disc(G(F))$.
\item The map $\JL$ commutes with $\rec_{G(F)}$ and $\rec_{\GL_n(F)}$,  \emph{i.e.} we have $\rec_{\GL_n(F)}(\JL(\pi))=\rec_{G(F)}(\pi)$ for all $\pi\in\Irr_\disc(G(F))$.
\item For $\pi\in\Irr(G(F))$,  let $\pi_1\times\pi_2\times\cdots\times\pi_r$ be the standard module which has a surjection onto $\pi$. 
Here,  $\pi_i\in\Irr_\disc(\GL_{m_i}(D))$ and $m=m_1+\cdots+m_r$ is a composition.
Then we have $\rec_{G(F)}(\pi)=\bigoplus_{i=1}^r \rec_{\GL_{m_i}(D)}(\pi_i)$.
\end{itemize}

We briefly recall the reformulation of the local Langlands correspondence for $G(F)$ by \cite{ABPS16}.
\begin{definition}\label{def:packet}
For $\phi\in\Phi(\GL_n(F))$,  we define the \emph{$L$-packet} attached to $\phi$ as
    \[
    \Pi_\phi=\bigcup_G \rec_{G(F)}^{-1}(\phi),
    \]
where $G$ runs over the all inner forms of $\GL_n(F)$.
This is a finite subset of $\bigcup_G\Irr(G(F))$.
We say that an $L$-packet $\Pi_\phi$ is \emph{generic} if it contains a generic representation of $\GL_n(F)$.
\end{definition}

For $\phi\in\Irr_\disc(\GL_n(F))$,  the size of the $L$-packet $\Pi_\phi$ is $n$ and satisfies $\JL(\pi_1)=\JL(\pi_2)$ for all $\pi_1,  \pi_2\in\Pi_\phi$.

\begin{definition}\label{def:S}
We define the \emph{component group} $S_\phi$ as
    \[
    S_\phi=Z_{\SL_n(\C)}(\phi)/Z_{\SL_n(\C)}(\phi)^\circ,
    \]
where $Z_{\SL_n(\C)}(\phi)$ is the centralizer of the image of $\phi$ by $\SL_n(\C)$ and $Z_{\SL_n(\C)}(\phi)^\circ$ is its identity component.
\end{definition}

Note that $S_\phi$ is isomorphic to $Z_{\SL_n(\C)}/(Z_{\SL_n(\C)}\cap Z_{\SL_n(\C)}(\phi)^\circ)$.
Here,  $Z_{\SL_n(\C)}$ is the center of $\SL_n(\C)$ which is the group of $n$-th roots of unity.
In particular,  $S_\phi$ is a quotient of $Z_{\SL_n(\C)}\cong\Z/n\Z$.

\begin{theorem}[\cite{ABPS16},  Theorem 2.2]\label{thm:ABPS}
For $G=\GL_m(D)$,  let $\chi_ G$ be the character of $Z_{\SL_n(\C)}$ which sends $\zeta_n:=\exp(2\pi\sqrt{-1}/n)$ to $h(D)$.
\setlength{\leftmargini}{25pt}
\begin{itemize}
\item[(1)] An $L$-parameter $\phi\in\Phi(\GL_n(F))$ lies in $\Phi(G(F))$ if and only if $\chi_G$ factors through $S_\phi$.
    \[
    \xymatrix{ 
    Z_{\SL_n(\C)} \ar^(0.6){\chi_G}[r] \ar @{->>}[d] 
    \ar @{}[dr]|(0.33)\circlearrowleft & \C^\times \\
    S_\phi \ar_{\chi_\pi} @{-->}[ur] &   }
    \]
If this is the case,  we write the character of $S_\phi$ induced from $\chi_G$ by $\chi_\pi$,  where $\pi$ denotes the unique element of $\Irr(G(F))$ which satisfies $\rec_{G(F)}(\pi)=\phi$,  \emph{i.e.} $\Irr(G(F))\cap\Pi_\phi=\{\pi\}$.
\item[(2)] The map $\pi\mapsto \chi_\pi$ defines a bijection from $\Pi_\phi$ to the set of characters of $S_\phi$.
\end{itemize}
\end{theorem}

\subsection{Linear periods}

Let $(G', H')$ be a symmetric pair,  \emph{i.e.} $G'$ is a reductive algebraic group over $F$ and $H'$ is the group of fixed points of an involution on $G'$.
We introduce the notion of tempered symmetric pairs,  which is called \emph{strongly discrete} in \cite{GO16}.
Following \cite[Section 2.7]{BP18},  we use the name of tempered symmetric pairs.

\begin{definition}
Let $A_{G'}$ be the maximal $F$-split torus in the center of $G'$ and set $A_{G'}^+=A_{G'}\cap H'$.
A unitary irreducible representation $\pi$ of $G'(F)$ is called \emph{$H'$-integrable} if all matrix coefficients of $\pi$ are integrable over $H/A_{G'}^+$.
We say that $(G',  H')$ is \emph{tempered} if all irreducible square integrable representations of $G'(F)$ are $H'$-integrable.
\end{definition}

\begin{lemma}
The symmetric pair $(G,  H)$ is tempered.
\end{lemma}

\begin{proof}
We freely use the notation and the terminology from \cite{GO16}.

First suppose that $d$ is odd.
The case of $d=1$ is \cite[Corollary 5.12]{GO16} and the general case follows from a similar argument.
Take $\delta\in E\setminus F$ such that $\delta^2\in F$ and the embedding of $E$ into $\M_m(D)$ as 
    \[
    a+b\delta\mapsto \diag\left(
        \begin{pmatrix}
        a & \delta^2 b \\
        b & a
        \end{pmatrix},  \ldots,  
        \begin{pmatrix}
        a & \delta^2 b \\
        b & a
        \end{pmatrix}    
        \right)
    \]
for $a,  b\in F$.
Note that $m$ is even since we assumed $n=md$ is even and $d$ is odd,  and hence this embedding is well-defined.
Let $t$ be the image of $\delta\in E^\times$ by this embedding and we define the involution $\theta$ on $G$ by $\theta(g)=tgt^{-1}$.
Then $H$ is the subgroup of the $\theta$-fixed points.
Let  $A_0=\{\diag(a_1,  \ldots,  a_m) \mid a_i\in F^\times\}$ be a maximal $F$-split torus in $G$.
Then $A_0^+:=A_0\cap H=\{\diag(t_1\mathbf{1}_2,  \ldots,  t_{m/2}\mathbf{1}_2) \mid t_i\in F^\times\}$ is a maximal $F$-split torus in $H$.
Let $M_1$ be the centralizer of $A_0^+$ in $G$,  which is a minimal $\theta$-stable Levi subgroup of $G$.

Since we have $\Sigma^{G/H}=\Sigma^H$ and $W^{G/H}=W^H$,  it follows from \cite[Corollary 5.4]{GO16} that we only need to show that $\rho_{G/H}$ is $M_1$-relatively weakly positive (see \cite[Definition 3.10]{GO16}).
For all $\alpha\in\Sigma^G$ such that $\alpha|_{A_0^+}\neq0$,  we have $\theta(\alpha)\neq\alpha$.
Hence we see $\rho_{G/H}=0$ by \cite[Lemma 4.6,  Proposition 4.7]{GO16}.
This proves that $(G,  H)$ is tempered.

Next we suppose that $d$ is even.
Take an embedding of $E$ into $D$ and let $E\hookrightarrow\M_m(D)$ be the embedding induced from $E\hookrightarrow D$.
Let $t$ be the image of $\delta$ by this embedding and we define the involution $\theta$ on $G$ by $\theta(g)=tgt^{-1}$.
Then $H$ is the subgroup of the $\theta$-fixed points.
Let  $A_0=\{\diag(a_1,  \ldots,  a_m) \mid a_i\in F^\times\}$ be a maximal $F$-split torus in $G$.
Then $A_0$ is also a maximal $F$-split torus in $H$.
Let $M_1$ be the centralizer of $A_0$ in $G$.

Since $\Sigma^{G/H}=\Sigma^G=\Sigma^H$ and $W^{G/H}=W^G=W^H$,  it follows from \cite[Corollary 5.4]{GO16} that we only need to show that $\rho_{G/H}$ is $M_1$-relatively weakly positive.
For each $\alpha\in\Sigma^{G/H}$,  let $L_\alpha^G$ (resp.\,$L_\alpha^H$) be the weight space of $\alpha$ in $\fg(F)$ (resp.\,$\fh(F)$).
Then $\dim_F(L_\alpha^G)=\dim_F(D)=d^2$ and $\dim_F(L_\alpha^H)=\dim_F(C)=\frac12d^2$,  where $C=\Cent_D(E)$ is the centralizer of $E$ in $D$.
By \cite[Lemma 4.6]{GO16} we have $m_{\theta,  \alpha}=2\dim_F(L_\alpha^H)-\dim_F(L_\alpha^G)=0$.
Hence we see that $\rho_{G/H}=0$ by \cite[Lemma 4.6,  Proposition 4.7]{GO16} and $(G,  H)$ is tempered.
\end{proof}

For $\pi\in\Irr(G(F))$,  note that $m(\pi,  \chi)=0$ unless $\omega_\pi=\chi^{\frac{n}{2}}|_{F^\times}$.
When $\chi=\mathbf{1}$ is the trivial character,  the multiplicity one for $(H,  \mathbf{1})$-periods is known by Guo \cite{Guo97} and Broussous and Matringe \cite{BM21}.

\begin{theorem}\label{thm:mult_one}
For $\pi\in\Irr(G(F))$,  we have $m(\pi,  \mathbf{1})\leq 1$.
\end{theorem}

Let 
    \[
    S(F)=\left\{s(g,  X):=
        \begin{pmatrix}
        g & gX \\
        \bzero & g
        \end{pmatrix}
    \,\middle|\,  g\in\GL_{n/2}(F),  \ X\in\M_{n/2}(F) \right\}
    \]
be the Shalika subgroup of $\GL_n(F)$.
We define the character $\Psi_{\psi,  \chi}$ of $S(F)$ by
    \[
    \Psi_{\psi,  \chi}(s(g,  X))=\chi(\det(g))\psi(\tr(X)),  \qquad g\in\GL_{n/2}(F) \ X\in\M_{n/2}(F),
    \]
where $\psi\,\colon F\to\C^\times$ is a fixed non-tirivial additive character of $F$.
We say that a representation $\pi$ of $\GL_n(F)$ has \emph{Shalika periods} with respect to $\chi|_{F^\times}$ if $\Hom_{S(F)}(\pi,  \Psi_{\psi,  \chi})\neq0$.
The multiplicity one for Shalika periods,  \emph{i.e.} 
    \[
    \dim\Hom_{S(F)}(\pi,  \Psi_{\psi,  \chi})\leq 1
    \]
 for $\pi\in\Irr(\GL_n(F))$ is known by Chen and Sun \cite[Theorem A]{CS20}.
 We will assume the following characterization of the Shalika periods.

\begin{hypothesis}\label{hyp:2}
Let $\pi\in\Irr_\disc(\GL_n(F))$ and set $\phi=\rec_{\GL_n(F)}(\pi)$.
Then $\pi$ has Shalika periods with respect to $\chi|_{F^\times}$ if and only if $\phi$ takes values in $\GSp_n(\C)$ with similitude factor $\chi|_{F^\times}$.
\end{hypothesis}
When $\chi=\mathbf{1}$ is the trivial character,  this is known. 
See \cite[Corollary 3.15,  Theorem 5.14]{Mat15}.

\subsection{Quasi-characters and strongly cuspidal functions}
\label{sec:quasi-characters}

We write the set of semisimple elements in $G(F)$ by $G_\semi(F)$.
For $x\in G_\semi(F)$,  let $G_x$ be the neutral component of the centralizer of $x$ in $G$ and $\fg_x(F)$ its Lie algebra.
Let $\fg_{x,  \reg}(F)$ be the set of regular semisimple elements in $\fg_x(F)$.

There is a notion of a \emph{good neighborhood} $\omega\subset \fg_x(F)$ of $0$.
It is an open neighborhood of $0$,  invariant under $G_x(F)$-conjugate and the exponential map is defined on it.
We refer the reader to \cite[Section 3.1]{Wal10} for the precise definition.

We fix a $G_x(F)$-invariant non-degenerate symmetric bilinear form $\langle \ ,  \ \rangle$ on $\fg_x(F)$.
We also fix a non-trivial additive character $\psi$ of $F$.
For $f\in\Cc(\fg_x(F))$,  its Fourier transform $\hat{f}\in\Cc(\fg_x(F))$ is defined as 
    \[
    \hat{f}(X)=\int_{\fg_x(F)} f(Y) \psi(\langle X,  Y\rangle) \rd Y,
    \]
where $\rd Y$ is the Haar measure on $\fg_x(F)$ normalized so that $\hat{\hat{f}}(X)=f(-X)$.

Let $\Nil(\fg_x(F))$ denote the set of nilpotent orbits in $\fg_x(F)$.
For $\cO\in\Nil(\fg_x(F))$,  there is a locally integrable function $\hat{j}(\cO,  X)$ on $\fg_x(F)$ such that
    \[
    J_\cO(\hat{f})=\int_{\fg_x(F)} f(X)\hat{j}(\cO,  X) \rd X
    \]
for all $f\in\cC(\fg_x(F))$.
Here,  $J_\cO(\hat{f})$ denotes the nilpotent orbital integral of $\hat{f}$ over $\cO$.

\begin{definition}
Let $\theta$ be a function defined almost everywhere on $G(F)$ which is invariant under $G(F)$-conjugation and its restriction defines a smooth function on $G_\reg(F)$.
The function $\theta$ is called a \emph{quasi-character} if for each $x\in G_\semi(F)$,  there is a good neighborhood $\omega_x\subset\fg_x(F)$ of $0$ and for each $\cO\in\Nil(\fg_x(F))$,  there is a constant $c_{\theta,  \cO}(x)\in\C$ such that
    \[
    \theta(x\exp(X))=\sum_{\cO\in\Nil(\fg_x(F))} c_{\theta,  \cO}(x)\hat{j}(\cO,  X),  \qquad X\in\omega_x\cap\fg_{x, \reg}(F).
    \]
\end{definition}

Let $\omega$ be a unitary character of $F^\times\cong Z_G(F)$.
We say that a quasi-character $\theta$ on $G(F)$ has central character $\omega$ if $\theta(zg)=\omega(z)\theta(g)$ for all $z\in Z_G(F)$ and $g\in G(F)$.
We denote by $\cC(G(F),  \omega)$ the space of Schwartz-Harish-Chadra functions $f\,\colon G(F)\to\C$ satisfying $f(zg)=\omega(z)f(g)$ for all $z\in Z_G(F)$ and $g\in G(F)$.
We refer to \cite[Section 2.3]{BP18} for details.

\begin{definition}
A function $f\in\cC(G(F),  \omega)$ is said to be \emph{strongly cuspidal} if we have
    \[
    \int_{U(F)} f(xu) \rd u=0
    \]
for all proper parabolic subgroup $P=MU$ of $G$ and all $x\in M(F)$.
We say that $f$ is \emph{cuspidal} if the above equality holds for all proper parabolic subgroup $P=MU$ of $G$ and all $x\in G(F)$.
\end{definition}

Let $x\in G_\reg(F)$ and set $M(x)=Z_G(A_{G_x})$.
Then $M(x)$ is a Levi subgroup of $G$,  which is minimal among those containing $x$.
For a strongly cuspidal function $f\in\cC(G(F),  \omega)$,  we set
    \[
    \theta_f(x)=(-1)^{\dim\left(\fa_{M(x)}^G\right)} D^G(x)^{-\frac12} J_{M(x)}(x,  f),  \qquad x\in G_\reg(F),
    \]
where $J_{M(x)}(x,  f)$ is the weighted orbital integral of $f$.
We refer the reader to \cite[Section 5.3]{Wal10} for the precise definition.
Then $\theta_f$ is a quasi-character by \cite[Corollary 5.9]{Wal10}.

\section{A reformulation of the conjecture}
\label{sec:reform}

\subsection{Statement}
\label{subsec:statement}

Let $n$ be an even positive integer.
For $\phi\in\Phi(\GL_n(F))$,  we set
    \[
    m(\Pi_\phi,  \chi) = \sum_{\pi\in\Pi_\phi}m(\pi,  \chi).
    \]

Suppose that $\phi$ takes values in $\GSp_n(\C)$ with similitude factor $\chi|_{F^\times}$.
Then it can be written in the form
    \begin{equation}\label{eq:L-param}
    \phi=\bigoplus_{i\in I_\phi}\phi_i \oplus 
    \bigoplus_{j\in J_\phi}(\phi_j \oplus (\phi_j^\vee\cdot\chi|_{F^\times})),  
    \end{equation}
with some finite index sets $I_\phi,  J_\phi$ and $\phi_k\in\Phi_\disc(\GL_{n_k}(F))$ for $k\in I_\phi\cup J_\phi$  which satisfy
\begin{itemize}
\item $n=\sum_{i\in I_\phi}n_i+2\sum_{j\in J_\phi}n_j$
\item for each $i\in I_\phi$,  $n_i$ is even and $\phi_i$ takes values in $\GSp_{n_i}(\C)$ with similitude factor $\chi|_{F^\times}$,
\item  for any $i,  i'\in I_\phi$,  $\phi_i\not\cong\phi_{i'}$ if $i\neq i'$.
\end{itemize}

Note that this decomposition is not unique but the set $\{\phi_i\}_{i\in I_\phi}$ is uniquely determined up to isomorphism.
We fix one such expression for each $L$-parameter which takes values in $\GSp_n(\C)$ with similitude factor $\chi|_{F^\times}$.
Then \cref{conj:PTB} can be rephrased as follows.

\begin{conjecture}\label{conj:reform}
Take an $L$-parameter $\phi\in\Phi(\GL_n(F))$.
\setlength{\leftmargini}{25pt}
\begin{itemize}
\item[(1)] If $m(\Pi_\phi,  \chi)\neq0$,  then $\phi$ takes values in $\GSp_n(\C)$ with similitude factor $\chi|_{F^\times}$.
\item[(2)] Suppose that $\phi$ takes values in $\GSp_n(\C)$ with similitude factor $\chi|_{F^\times}$ and write it as \eqref{eq:L-param}.
We also assume that $\Pi_\phi$ is generic.
Then for $\pi\in\Pi_\phi$,  $m(\pi,  \chi)\neq0$ if and only if 
    \begin{equation}\label{eq:reform}
    \varepsilon\left(\phi_i\otimes\Ind_{W_E}^{W_F}(\chi^{-1})\right)
    =\left(\chi_\pi(\zeta_n)\chi\eta_{E/F}(-1)\right)^{\frac{n_i}{2}}
    \end{equation}
for each $i\in I_\phi$.
\end{itemize}
\end{conjecture}

\begin{remark}Keep the notation of \cref{conj:reform}.
\setlength{\leftmargini}{25pt}
\begin{itemize}
\item[(1)] \cref{conj:reform} is a refinement of the original one because it provides a necessary and sufficient condition while \cref{conj:PTB} concerns only one direction.

\item[(2)] Applying \cite[(30.4.2),  (30.4.3)]{BH06}, the equation \eqref{eq:reform} becomes
    \[
    \varepsilon(\bc_{E/F}(\phi_i)\otimes\chi^{-1})
    =\left(\chi_\pi(\zeta_n)\chi(-1)\right)^{\frac{n_i}{2}},
    \]
where $\bc_{E/F}(\phi_i)$ denotes the restriction of $\phi_i$ to $WD_E$,  which belongs to $\Phi(\GL_{n_i}(E))$.
Note also that $\phi$ takes values in $\GSp_n(\C)$ with similitude factor $\chi|_{F^\times}$ if and only if $\bc_{E/F}(\phi)$ takes values in $\GSp_n(\C)$ with similitude factor $\chi\circ\Nm_{E/F}$.
\end{itemize}
\end{remark}

\subsection{Comparison with the original conjecture}
\label{subsec:comparison}

We examine the relation between \cref{conj:PTB} and \cref{conj:reform}.
Take an $L$-parameter $\phi\in\Phi(\GL_n(F))$.
Suppose that $\phi$ takes values in $\GSp_n(\C)$ with similitude factor $\chi|_{F^\times}$ and write it in the form \eqref{eq:L-param}.
We also assume that $\Pi_\phi$ is generic.

Let $\pi\in\Pi_\phi$ be a representation of $\GL_m(D)$,  where $D$ is a central division algebra over $F$ of dimension $\dim_F(D)=d^2$.
Note that $m$ and $d$ satisfy $n=md$.
Since $\Pi_\phi$ is generic,  $\pi$ is isomorphic to a standard module $\widetilde{\pi}=\pi_1\times\pi_2\times\cdots\times\pi_r$ with some composition $m=m_1+\cdots+m_r$ and $\pi_i\in\Irr_\disc(\GL_{m_i}(D))$.
We set $n_i=m_id$ and $\phi_i=\rec_{\GL_{m_i}(D)}(\pi_i)\in\Phi(\GL_{n_i}(F))$ for each $i$.

\begin{lemma}\label{lem:comparison}
Keep the above notation.
\cref{conj:reform} (2) is equivalent to the following assertion.
We have $m(\pi,  \chi)\neq0$ if and only if there is an involutive permutation $\varsigma\in\fS_r$ such that 
\begin{itemize}
\item $m_{\varsigma(i)}=m_i$ for each $i$.
\item if $\varsigma(i)=i$,  then $n_i$ is even,  $\phi_i$ takes values in $\GSp_{n_i}(\C)$ with similitude factor $\chi|_{F^\times}$ and 
        $\varepsilon\left(\phi_i\otimes\Ind_{W_E}^{W_F}(\chi^{-1})\right)=(-1)^{m_i}\chi\eta_{E/F}(-1)^\frac{n_i}{2}$.
\item if $\varsigma(i)\neq i$,  then $\pi_{\varsigma(i)}\cong \pi_i^\vee\cdot\chi_{\GL_{m_i}(D)}$.
\end{itemize}
\end{lemma}

\begin{proof}
Suppose there is an involutive permutation $\varsigma\in\fS_r$ which satisfies the above three conditions.
We prove that $\phi$ satisfies the conditions of \cref{conj:reform} (2).

Note that $\pi_{\varsigma(i)}\cong \pi_i^\vee\cdot\chi_{\GL_{m_i}(D)}$ is equivalent to $\phi_{\varsigma(i)}\cong\phi_i^\vee\cdot\chi|_{F^\times}$.
Hence $\phi$ takes values in $\GSp_n(\C)$ with similitude factor $\chi|_{F^\times}$.
If $\varsigma(i)=i$,  then $\phi_i\cong\phi_i^\vee\cdot\chi|_{F^\times}$ since $\phi_i$ takes values in $\GSp_{n_i}(\C)$ with similitude factor $\chi|_{F^\times}$.
Hence if there are two distinct fixed points $i,  j$ of $\varsigma$ such that $\phi_i\cong\phi_j$,  we can replace $\varsigma$ with the involutive permutation $\varsigma'\in\fS_r$ given by
    \[
    \varsigma'(k)=
        \begin{cases}
        j & \text{if $k=i$} \\
        i & \text{if $k=j$} \\
        \varsigma(k) & \text{otherwise}.
        \end{cases}
    \]
Repeating this procedure,  we may assume that for any fixed points $i,  j$ of $\varsigma$,  $\phi_i\not\cong\phi_j$. 

Let $I_\phi$ be the set of fixed points of $\varsigma$.
Take a subset $J_\phi\subset\{1,  \ldots, r\}\setminus I_\phi$ so that $J_\phi\cap\varsigma(J_\phi)=\emptyset$ and $\{1,  \ldots, r\}=I_\phi\cup J_\phi \cup\varsigma(J_\phi)$.
Then we obtain the decomposition
    \[
    \phi=\bigoplus_{i\in I_\phi}\phi_i \oplus 
    \bigoplus_{j\in J_\phi}(\phi_j \oplus (\phi_j^\vee\cdot\chi|_{F^\times})),  
    \]
satisfying the conditions of \eqref{eq:L-param}.

The Hasse invariant $h(D)$ of $D$ is a primitive $d$-th root of unity.
We write it as $h(D)=\zeta_d^k$ with some positive integer $k<d$ prime to $d$.
For $i\in I_\phi$,  we have
    \[
    \chi_\pi(\zeta_n)^{\frac{n_i}{2}}=h(D)^{\frac{n_i}{2}}=\zeta_d^{k\cdot\frac{m_id}{2}}
    =(\zeta_n^{\frac{n}{d}})^{k\cdot\frac{m_id}{2}}=(-1)^{km_i}.
    \]
If $k$ is odd,  this equals $(-1)^{m_i}$. 
If $k$ is even,  then $d$ is odd and $m_i$ is even.
Hence $(-1)^{km_i}=1=(-1)^{m_i}$.
In any case,  we have 
    \[
    \left(\chi_\pi(\zeta_n)\chi\eta_{E_F}(-1)\right)^{\frac{n_i}{2}}=(-1)^{m_i}\chi\eta_{E/F}(-1)^{\frac{n_i}{2}}.
    \]

Conversely,  suppose that $\phi$ takes values in $\GSp_n(\C)$ with similitude factor $\chi|_{F^\times}$ and is decomposed as \eqref{eq:L-param}.
For each $j\in J_\phi$,  take $\varsigma(j)\in\{1,  \ldots,  r\}$ so that $\phi_j^\vee\cdot\chi|_{F^\times}\cong\phi_{\varsigma(j)}$.
We set $\varsigma(\varsigma(j))=j$ for $\varsigma(j)\in\varsigma(J_\phi)$.
For $i\in I_\phi$,  we define $\varsigma(i)=i$.
Then $\varsigma$ becomes an involutive element in $\fS_r$.
From the above computation of $\chi_\pi(\zeta_n)^{\frac{n_i}{2}}$ for $i\in I_\phi$,  we see that $\varsigma$ satisfies the desired conditions.
\end{proof}

The next proposition is a consequence of the Mackey theory.

\begin{proposition}\label{prop:reduction_disc}
Keep the above notation.
The representation $\pi$ is $(H,  \chi_H)$-distinguished if and only if there is an involutive permutation $\varsigma\in\fS_r$ such that 
\begin{itemize}
\item $m_{\varsigma(i)}=m_i$ for each $i$.
\item if $\varsigma(i)=i$,  then $n_i$ is even and $\pi_i$ is $(H_i,  \chi_{H_i})$-distinguished.
Here,  $H_i$ denotes the centralizer of $E^\times$ in $\GL_{m_i}(D)$.
\item if $\varsigma(i)\neq i$,  $\pi_{\varsigma(i)}\cong \pi_i^\vee\cdot\chi_{\GL_{m_i}(D)}$.
\end{itemize}
\end{proposition}

\begin{proof}
First suppose $F$ is a finite extension of $\Q_p$.
If $\chi=\mathbf{1}$ is the trivial character,  this is \cite[Theorem 1.3]{Suz21}.
The argument remains valid in the general case with minor modifications.
\end{proof}

Comparing the conditions in \cref{lem:comparison} with those in \cref{prop:reduction_disc},  we obtain the next corollary.

\begin{corollary}\label{cor:reduction_disc}
\cref{conj:PTB} for essentially square integrable representations implies \cref{conj:reform}.
\end{corollary}

In particular, \cref{thm:known_cases} and \cref{cor:reduction_disc} imply the following.

\begin{corollary}
\cref{conj:reform} holds when $p\neq 2$ and $\chi=\mathbf{1}$ is the trivial character.
\end{corollary}

\begin{remark}\label{rem:archimedean}
We briefly discuss the case where $F$ is archimedean.
Suppose that $F=\R$ and $E=\C$.
Note that $\eta_{E/F}=\sgn$ is the sign character of $\R^\times$.

Let $\H$ be the quaternion division algebra.
Then we have $G=\GL_n(\R)$ or $G=\GL_{n/2}(\H)$.
We define the character $\chi_G$ of $Z_{\SL_n(\C)}$ as in \cref{thm:ABPS}.
Since $h(\R)=1$ and $h(\H)=-1$,  the character $\chi_G$ is trivial if $G=\GL_n(\R)$ and the unique character of order two if $G=\GL_{n/2}(\H)$.

Also we define the $L$-packet $\Pi_\phi$ and the component group $S_\phi$ for $\phi\in\Phi(\GL_n(\R))$ as in \cref{def:packet} and \cref{def:S},  respectively.
Then \cref{thm:ABPS} remains valid according to \cite[Theorem 2.2]{ABPS16}.
In particular,  we have 
    \[
    S_\phi = 
        \begin{cases}
        \Z/2\Z &  \text{if $\phi$ is relevant to $\GL_{n/2}(\H)$,}  \\
         \{1\}  & \text{otherwise.}
         \end{cases}
    \]

The equation \eqref{eq:reform} becomes
    \[
    \varepsilon\left(\phi_i\otimes\Ind_{W_E}^{W_F}(\chi^{-1})\right)
    = \epsilon(G)^{\frac{n_i}{2}}\chi(-1)^{\frac{n_i}{2}},  
    \]
where we set $\epsilon(G)=-1$ if $G=\GL_n(\R)$ and $\epsilon(G)=1$ if $G=\GL_{n/2}(\H)$.
This coincides with the root number condition considered in \cite{ST23}.
It follows that the ``only if'' part of \cref{conj:reform} is nothing but \cite[Theorem 1.1]{ST23}.
On the other hand,  the ``if'' part of \cref{conj:reform} follows from \cite[Theorem 5.4,  Remark 5.6]{MOY24}.
Therefore,  \cref{conj:reform} holds for $F=\R$.
\end{remark}

\section{The multiplicity formula}
\label{sec:muult}

In this section,  we discuss a relation between \cref{conj:reform} and the multiplicity formula predicted by Wan \cite{Wan22}.
First we recall the definition of the geometric multiplicity in our setting,  which is defined in \cite[Definition 7.1]{Wan22} for general spherical varieties. 

Set $\omega=\chi^{\frac{n}{2}}|_{F^\times}$ and we regard it as a character of $F^\times\cong Z_G(F)$. 
Let $\theta$ be a quasi-character on $G(F)$ with central chracter $\omega$.
We fix a set $\cT_\elliptic(H)$ of representatives of $H(F)$-conjugacy classes of maximal elliptic tori in $H$.

First, consider the case $m$ is odd.
Note that in this case $E$ embeds in $D$.
Let $C=\Cent_D(E)$ be the centralizer of $E$ in $D$.
This is a central division algebra over $E$ of dimension $\dim_E(C)=\tfrac{d^2}{4}$ and $H\cong\GL_m(C)$.
\begin{definition}\label{def:geom_mult1}
Suppose that $m$ is odd.
Set
    \[
    m_\geom(\theta,  \chi)=\sum_{T\in\cT_\elliptic(H)}|W(H,  T)|^{-1}\int_{Z_G(F)\bs T(F)}D^H(t)\theta(t)\chi_H^{-1}(t) \rd t.
    \]
For $\pi\in\Irr(G(F))$ with central chracter $\omega_\pi=\omega$,  we set 
    \[
    m_\geom(\pi,  \chi)=m_\geom(\theta_\pi,  \chi).
    \]
\end{definition}

Next,  suppose that $m$ is even.
In this case $E$ embeds in $\M_m(D)$ as
    \[
    E\hookrightarrow \M_m(D)\,; \ a+b\delta\mapsto 
        \begin{pmatrix}
        a\cdot\mathbf{1}_{m/2} & \delta^2 b\cdot\mathbf{1}_{m/2} \\
        b\cdot\mathbf{1}_{m/2} & a\cdot\mathbf{1}_{m/2}
        \end{pmatrix},  \quad a,  b\in F,
    \]
where $\delta$ is an element of $E\setminus F$ such that $\delta^2\in F$.
Under this embedding,  the centralizer $\Cent_{\M_m(D)}(E)$ of $E$ in $\M_m(D)$ consists of
    \[
        h(A, B):=\begin{pmatrix}
        A & \delta^2 B \\
        B & A
        \end{pmatrix},  \quad A,  B\in \M_{m/2}(D)
    \]
and $H= G\cap\Cent_{\M_m(D)}(E)$.

Let $H'$ be the subgroup of $H$ consisting of $h(A,  \mathbf{0}_{m/2})$ with $A\in\GL_{m/2}(D)$.
For a maximal elliptic torus $T'$ of $H'$,  there is an extension $K/F$ of degree $n/2$ such that $T'(F)\cong K^\times$.
For $t'\in T'(F)\cap G_\reg(F)$,  let $G_{t'}$ (resp.\,$H_{t'}$) be the centralizer of $t'$ in $G$ (resp.\,$H$).
Then $G_{t'}(F)\simeq\GL_2(K)$ and $H_{t'}(F)\cong(K\otimes_F E)^\times$.
Let $\fg_{t'}$ be the Lie algebra of $G_{t'}$  and $\cO_{t'}\in\Nil(\fg_{t'}(F))$ the unique regular nilpotent orbit in $\fg_{t'}(F)$.

We also fix a set $\cT_\elliptic(H')$ of representatives of $H'(F)$-conjugacy classes of maximal elliptic tori in $H'$.

\begin{definition}\label{def:geom_mult2}
Suppose that $m$ is even.
Set
    \[
    \begin{aligned}
    m_\geom(\theta,  \chi)=&\sum_{T\in\cT_\elliptic(H)}|W(H,  T)|^{-1}
    \int_{Z_G(F)\bs T(F)}D^H(t)\theta(t)\chi_H^{-1}(t) \rd t \\ 
    &\quad +\sum_{T'\in\cT_\elliptic(H')}|W(H', T')|^{-1}
    \int_{Z_G(F)\bs T'(F)}D^{H'}(t')^2 c_{\theta,  \cO_{t'}}(t')\chi_H^{-1}(t') \rd t'.
    \end{aligned}
    \]
For $\pi\in\Irr(G(F))$ with central chracter $\omega_\pi=\omega$,  we set 
    \[
    m_\geom(\pi,  \chi)=m_\geom(\theta_\pi,  \chi).
    \]
\end{definition}

The convergence is assured by the next lemma.

\begin{lemma}
The integrals
    \[
    \int_{Z_G(F)\bs T(F)}D^H(t)\theta(t)\chi_H^{-1}(t) \rd t,  \qquad 
    \int_{Z_G(F)\bs T'(F)}D^{H'}(t')^2 c_{\theta,  \cO_{t'}}(t')\chi_H^{-1}(t') \rd t'
    \]
converge absolutely for a quasi-character $\theta\in C^\infty(G_\reg(F))$,   $T\in\cT_\elliptic(H)$ and $T'\in\cT_\elliptic(H')$. 
\end{lemma}

\begin{proof}
The absolute convergence of the second integral is taken care of by \cite[Lemma 3.2]{BPW19}.
It follows from \cite[Lemma B.1.2]{BP20} that $(D^G)^{-\frac12}D^H$ is integrable over $Z_G(F)\bs T(F)$ for each $T\in\cT_\elliptic(H)$.
Since $(D^G)^\frac12\theta$ is locally bounded,  the  first integral converges absolutely.
\end{proof}

The following is \cite[Conjecture 7.4 (2)]{Wan22} for linear periods.

\begin{conjecture}\label{conj:mult_formula}
The multiplicity formula
    \begin{equation}\label{eq:mult_formula}
    m(\pi,  \chi)=m_\geom(\pi,  \chi)
    \end{equation}
holds for all $\pi\in\Irr_\disc(G(F))$.
\end{conjecture}

The author is indebted to Chen Wan who suggested the proof of the next proposition.

\begin{proposition}\label{prop:deduce}
Assume \cref{hyp:2} and \cref{conj:mult_formula} hold.
For $\phi\in\Phi_\disc(\GL_n(F))$,  
    \[
    m(\Pi_\phi,  \chi)=
        \begin{cases}
        n/2 & \text{if $\phi$ takes values in $\GSp_n(\C)$ with similitude factor $\chi|_{F^\times}$},  \\
        0 & \text{otherwise}.
        \end{cases}
    \]
\end{proposition}

\begin{proof}
For $\pi_1,  \pi_2\in\Pi_\phi$,  let $G_1$,  $G_2$ be the inner forms of $\GL_n(F)$ such that $\pi_i\in\Irr_\disc(G_i(F))$,  $i=1,  2$.
We can write the character relation of \cref{thm:DKV} as 
    \begin{equation}\label{eq:DKV}
    \chi_{\pi_1}(-1)\theta_{\pi_1}(t_1)=\chi_{\pi_2}(-1)\theta_{\pi_2}(t_2)
    \end{equation}
for  $t_1\in G_{1,  \reg}(F)$ and $t_2\in G_{2,  \reg}(F)$ with $t_1 \leftrightarrow t_2$.
Here,  $t_1 \leftrightarrow t_2$ means that $t_1$ and $t_2$ have the same characteristic polynomial.

Suppose that $\chi_{\pi_1}(-1)=\chi_{\pi_2}(-1)=-1$.
By the Wely's integration formula,  we can rewrite the multiplicity formula \eqref{eq:mult_formula} as 
   \begin{align*}
   m(\pi_1,  \chi)&=\int_{\Gamma_\elliptic(H_1/Z_{G_1})}D^{H_1}(t)\theta_{\pi_1}(t)\chi_{H_1}^{-1}(t) \rd t, \\
   m(\pi_2,  \chi)&=\int_{\Gamma_\elliptic(H_2/Z_{G_2})}D^{H_1}(t)\theta_{\pi_2}(t)\chi_{H_2}^{-1}(t) \rd t,
   \end{align*}
where $H_i=\Cent_{G_i}(E^\times)$ is the symmetrric subgroup of $G_i$ and $\Gamma_\elliptic(H_i/Z_{G_i})$ is the set of elliptic conjugacy classes in $(H_i/Z_{G_i})(F)$.
The correspondence $t_1\leftrightarrow t_2$ above induces a bijection between $\Gamma_\elliptic(H_1/Z_{G_1})$ and $\Gamma_\elliptic(H_2/Z_{G_2})$ which we also write by $t_1\leftrightarrow t_2$.
Since $D^{H_1}(t_1)=D^{H_2}(t_2)$,  $\theta_{\pi_1}(t_1)=\theta_{\pi_2}(t_2)$,  and $\chi_{H_1}^{-1}(t_1)=\chi_{H_2}^{-1}(t_2)$ for all pairs $t_1\leftrightarrow t_2$,  we obtain $m(\pi_1,  \chi)=m(\pi_2,  \chi)$.

The case where $\chi_{\pi_1}(-1)=\chi_{\pi_2}(-1)=1$ is similar.
Let $\Gamma_\elliptic(H'_i/Z_{G_i})$ be the set of elliptic conjugacy classes in $(H'_i/Z_{G_i})(F)$.
Since both $H'_1$ and $H'_2$ are isomorphic to inner forms of $\GL_{n/2}(F)$,  there is a natural bijection between $\Gamma_\elliptic(H'_1/Z_{G_1})$ and $\Gamma_\elliptic(H'_2/Z_{G_2})$ which we also write by $t'_1\leftrightarrow t'_2$.
Note that we have $c_{\pi_1,  \cO_{t'_1}}(t'_1)=c_{\pi_2,  \cO_{t'_2}}(t'_2)$ for $t'_1 \leftrightarrow t'_2$,  which is shown in the proof of \cite[Theorem 6.1]{BPW19}.
Hence the equality $m(\pi_1,  \chi)=m(\pi_2,  \chi)$ follows from the definition of the geometric multiplicity and the multiplicity formula \eqref{eq:mult_formula}.

Suppose that $\chi_{\pi_1}(-1)=-\chi_{\pi_2}(-1)=1$.
From \cref{def:geom_mult1} and \cref{def:geom_mult2},  the multiplicity formula \eqref{eq:mult_formula} and the character relation \eqref{eq:DKV},  we obtain
    \[
    m(\pi_1,  \chi)+m(\pi_2,  \chi)
    =\int_{\Gamma_\elliptic(H'_1/Z_{G_1})}
    D^{H'_1}(t')^2 c_{\pi,  \cO_{t'}}(t')\chi_{H_1}^{-1}(t') \rd t'.
    \]
By \cite[Theorem 1.4]{BPW19},  the right hand side is the multiplicity of the generalized Shalika model of $\pi_1$.
Let $\pi_0\in\Pi_\phi$ be the representation of $\GL_n(F)$.
By \cite[Theorem 6.1]{BPW19},  $\pi_1$ has a non-zero generalized Shalika model if and only if so does $\pi_0$.
Since we assume \cref{hyp:2},  we obtain
    \begin{equation}\label{eq:deduce2}
    m(\pi_1,  \chi)+m(\pi_2,  \chi)=
        \begin{cases}
        1 & \text{if $\phi$ takes values in $\GSp_n(\C)$ with similitude factor $\chi|_{F^\times}$} \\ 
        0 & \text{otherwise}.
        \end{cases}
    \end{equation}
The same argument applies to the case that $\chi_{\pi_1}(-1)=-\chi_{\pi_2}(-1)=-1$.
\end{proof}

In the case $\chi$ is the trivial character,  \cref{conj:reform} follows from \cref{conj:mult_formula}.

\begin{corollary}\label{cor:trivial}
\cref{conj:mult_formula} for the trivial character $\chi=\mathbf{1}$ implies \cref{conj:reform} for $\chi=\mathbf{1}$.
\end{corollary}

\begin{proof}
By \cite[Theorem 1.4]{Suz21},  the problem reduces to the case of essentially square integrable representations.
Take $\phi\in\Phi_\disc(\GL_n(F))$.
Let $\Pi_{\phi,  \mathrm{dist}}$ be the subset of distinguished elements in $\Pi_\phi$ and set
    \[
    \Pi_{\phi,  \varepsilon}=\{\pi\in\Pi_\phi \mid \varepsilon(\bc_{E/F}(\phi))=\chi_\pi(-1)\}.
    \]
Note that $|\Pi_{\phi,  \varepsilon}|=n/2$.
The ``only if'' part of \cref{conj:reform} (2) for discrete $L$-parameters is proved by Xue \cite{Xue21}.
Hence we have $\Pi_{\phi,  \mathrm{dist}}\subset\Pi_{\phi,  \varepsilon}$.

On the other hand,  $|\Pi_{\phi,  \mathrm{dist}}|=m(\Pi_\phi,  \mathbf{1})$ from \cref{thm:mult_one}.
Recall that \cref{hyp:2} is proved by Matringe \cite{Mat15} when $\chi=\mathbf{1}$. 
Hence it follows from \cref{prop:deduce} that $\Pi_{\phi,  \varepsilon}$ coincides with the set of distinguished elements in $\Pi_\phi$ when $\phi$ takes values in $\Sp_n(\C)$ and $\Pi_{\phi,  \mathrm{dist}}=\emptyset$ otherwise.
This finishes the proof.
\end{proof}

\subsection{A local relative trace formula}
\label{subsec:RTF}

Let $f\in\cC(G(F),  \omega)$ be a strongly cuspidal function.
We define the kernel function $K_f$ by
    \[
    K_f(g_1,  g_2)=\int_{Z_G(F)\bs H(F)}f(g_1^{-1}hg_2)\chi_H^{-1}(h) \rd h,  \qquad g_1,  g_2\in G(F).
    \]
Since the symmetric pair $(G,  H)$ is tempered,  this integral is absolutely convergent.
The distribution $I(f)$ is defined as
    \[
    I(f)=\int_{H(F)\bs G(F)}K_f(g,  g)\rd g.
    \]
Again since $(G,  H)$ is tempered,  from \cite[Theorem 4.1.1]{BP18},  the defining integral of $I(f)$ converges absolutely.

We define the distribution $I_\spec(f)$,  the spectral side of $I(f)$ by
    \[
    I_\spec(f)=\sum_{\substack{\pi\in\Irr_\disc(G(F)) \\ \omega_\pi=\omega}}
    m(\pi,  \chi)\theta_{\pi^\vee}(f).
    \]
When $\chi$ is a unitary character and $f\in\cC(G,  \omega)$ is cuspidal,  we have the spectral expansion $I(f)=I_\spec(f)$,  see \cite[Theorem 3.1.1]{BP18}.

We define the distribution $I_\geom(f)$,  the geometric side of $I(f)$ by $I_\geom(f)=m_\geom(\theta_f,  \chi)$.
The geometric expansion of the local relative trace formula \cite[Conjecture 7.10 (1)]{Wan22} is as follows.

\begin{conjecture}\label{conj:geom_exp}
The geometric expansion 
    \[
    I(f)=I_\geom(f)
    \]
 holds for a strongly cuspidal function $f\in\cC(G(F),  \omega)$.
\end{conjecture}

\begin{proposition}\label{prop:geom_exp}
\cref{conj:geom_exp} implies \cref{conj:mult_formula}.
\end{proposition}

\begin{proof}
Twisting $\pi$ and $\chi$ by a real unramified character,  we may suppose that $\chi$ and the central character of $\pi$ are unitary.

Let $\pi\in\Irr_\disc(G(F))$ with central character $\omega_\pi=\omega$.
Take a matrix coefficient $f$ of $\pi$ so that $f(1)\neq0$.
Then $f\in\cC(G(F),  \omega)$ is cuspidal and we have \cite[Proposition 2.6.1]{BP18}
    \[
    \theta_f(x)=d(\pi)^{-1}f(1)\theta_\pi(x)
    \]
for all $x\in G_\reg(F)$.
Here,  $d(\pi)$ is the formal degree of $\pi$.
From this,  we also obtain
    \[
    c_{f,  \cO_{t'}}(h')=d(\pi)^{-1}f(1)c_{\pi,  \cO_{t'}}(h')
    \]
for all $h'\in H'_\reg(F)$,  $T'\in\cT_\elliptic(H')$ and $t'\in T'(F)\cap G_\reg(F)$.
Hence we have $I_\geom(f)=d(\pi)^{-1}f(1)m_\geom(\pi,  \chi)$.

On the other hand,  by the Schur orthogonality relations,  we obtain
    \[
    I_\spec(f)=m(\pi,  \chi)\theta_{\pi^\vee}(f)=d(\pi)^{-1}f(1)m(\pi,  \chi),
    \]
where the second equality is a direct consequence of the definition of $d(\pi)$.
Now the multiplicity formula \eqref{eq:mult_formula} follows from \cref{conj:geom_exp}.
\end{proof}

\section{An expression for the root number}
\label{sec:root_number}

Recall that $E/F$ is a quadratic extension of non-archimdean local fields of characteristic zero.
Let $c$ be the non-trivial $F$-automorphism of $E$.
We write the norm and trace map by $\Nm_{E/F}$ and $\tr_{E/F}$,  respectively.
Fix a non-tirivial additive character $\psi$ of $F$ and set $\psi_E=\psi\circ\tr_{E/F}$.

\subsection{Base change lift}
\label{subsec:base_change}

In this subsection,  we may remove the condition that $n$ is even.
Let $\bfG=\Res_{E/F}\GL_n$,  where $\Res_{E/F}$ denotes the Weil restriction of scalars and $Z_\bfG$ be the center of $\bfG$.

We define an involution $\tau$ of $\bfG(F)$ by $\tau(g)=wg^cw^{-1}$,  where
    \[
    w=\renewcommand{\arraystretch}{0.8}
        \begin{pmatrix}
        &&1 \\
        &\idots& \\
        1&&
        \end{pmatrix}
    \in \bfG(F).
    \]
For a representation $\pi'$ of $\bfG(F)$,  we define representations $\pi'^c$ and $\pi'\circ\tau$ of $\bfG(F)$ on the same representation space as $\pi'$,  by putting $\pi'^c(g)=\pi'(g^c)$ and $\pi'\circ\tau(g)=\pi'(\tau(g))$.
If $\pi'$ is irreducible,  let $\omega_{\pi'}$ denote its central character.

Let $\bfN$ be the subgroup of upper unipotent matrices of $\bfG$.
We write by $\psi_E$ the non-degenerate character $u\mapsto\psi_E(\sum_{i=1}^{n-1}u_{i,i+1})$ of $\bfN(F)$.
Let $\Irr_\gen(\bfG(F))$ be the set of generic irreducible representations of $\bfG(F)$. 
For $\pi'\in\Irr_\gen(\bfG(F))$,  we denote its Whittaker model by $\cW(\pi',  \psi_E)$.
Then $\bfG(F)$ acts on $\cW(\pi',  \psi_E)$ via right translation,  which we write by $R$.
For $W\in\cW(\pi',  \psi_E)$,  we define the operator $\tilde{\pi}'(\tau)$ by
    \[
    (\tilde{\pi}'(\tau)W)(x)=W(x^cw),  \qquad x\in \bfG(F).
    \]
Then $R(g)(\tilde{\pi}'(\tau)W)(x)=W(x^cg^cw)=W(x^cw\tau(g))=\tilde{\pi}'(\tau)((R\circ\tau)(g)W)(x)$ for $g\in \bfG(F)$.

Fix a non-zero elements $\phi\in\Hom_N(\pi',  \psi_E)$ and $\phi^\vee\in\Hom_N(\pi'^\vee,  \psi_E^{-1})$.
For $v\in\pi'$ and $v^\vee\in\pi'^\vee$,  we define $W_v\in\cW(\pi',  \psi_E)$ and $W_{v^\vee}\in\cW(\pi'^\vee,  \psi_E^{-1})$ by $W_v(g)=\phi(\pi'(g)v)$ and $W_{v^\vee}(g)=\phi^\vee(\pi'^\vee(g)v^\vee)$ for $g\in \bfG(F)$.

By the uniqueness of Whittaker models,  we have 
    \[
    \cW(\pi'\circ\tau,  \psi_E)=\{ \tilde{\pi}'(\tau)W \mid W\in\cW(\pi',  \psi_E)\}.
    \]
Suppose that $\pi'\cong\pi'^c$,  or equivalently  $\pi'\cong\pi'\circ\tau$.
Then $\cW(\pi',  \psi_E)=\cW(\pi'\circ\tau,  \psi_E)$ and $\tilde{\pi}'(\tau)$ is an involution on $\cW(\pi',  \psi_E)$.

Let $\bfG^+=\bfG\rtimes\{1,  \tau\}$ be a non-connected algebraic group and $\tilde{\bfG}=\bfG\tau$ its connected component.
For $\pi'\in\Irr_\gen(\bfG(F))$ satisfying $\pi'\cong\pi'^c$,  we extend it to a representation $\pi'^+$ of $\bfG^+(F)$ by setting $\pi'^+(g\tau)=\pi'(g)\tilde{\pi}'(\tau)$ for $g\in\bfG(F)$.
Let $\tilde{\pi}'=\pi'^+|_{\tilde{\bfG}(F)}$ be a twisted representation of $\tilde{\bfG}(F)$ in the sense of \cite{Wal12b}.

We summarize the facts about base change lifting for representations of $\GL_n(F)$.
For details,  see \cite[Chapter 1]{AC89}.

For each $\tilde{g}\in\tilde{\bfG}(F)$,  $\tilde{g}^2$ is $\bfG(F)$-conjugate to an element of $\GL_n(F)$,  which is unique up to $\GL_n(F)$-conjugate.
We write such an element of $\GL_n(F)$ by $\cN(\tilde{g})$.
We say that $\tilde{g}\in\tilde{\bfG}(F)$ is regular semisimple (resp.\,elliptic) if $\cN(\tilde{g})\in\GL_{n,  \reg}(F)$ (resp.\,$\cN(\tilde{g})\in\GL_{n,  \elliptic}(F)$). 
Let $\tilde{\bfG}_\reg(F)$ and (resp.\,$\tilde{\bfG}_\elliptic(F)$) be the set of regular semisimple (resp.\,elliptic) elements in $\tilde{\bfG}(F)$.
Note that $\tilde{\bfG}_\reg(F)\tau^{-1}$ (resp.\,$\tilde{\bfG}_\elliptic(F)\tau^{-1}$) is the set of $c$-regular semisimple (resp.\,$c$-elliptic) elements in $\bfG(F)$ in the sense of \cite{AC89}.
For a representation $\pi'$ of $\bfG(F)$ satisfying $\pi'\cong\pi'^c$,  the twisted character $\Theta_{\tilde{\pi}'}$ is defined by $\Theta_{\tilde{\pi}'}(\tilde{g})=\tr(\tilde{\pi}'(\tilde{g}))$ for $\tilde{g}\in\tilde{\bfG}_\reg(F)$.

\begin{theorem}[\cite{AC89} Chapter 1,  Theorem 6.2]
There is a unique map $\bc_{E/F}\,\colon \Irr_\disc(\GL_n(F))\to\Irr(\bfG(F))$ such that for $\pi\in\Irr_\disc(\GL_n(F))$,  $\pi':=\bc_{E/F}(\pi)$ is an essentially tempered representation which satisfies $\pi'\cong\pi'^c$ and 
    \begin{equation}\label{eq:AC}
    \Theta_{\tilde{\pi}'}(\tilde{g})=\theta_\pi(\cN(\tilde{g})),  \qquad \tilde{g}\in\tilde{\bfG}_\reg(F).
    \end{equation}
\end{theorem}

For $\pi\in\Irr_\disc(\GL_n(F))$,  set $\phi_\pi=\rec_{\GL_n(F)}(\pi)$,  $\pi'=\bc_{E/F}(\pi)$ and $\phi_{\pi'}=\rec_{\bfG(F)}(\pi')$.
By \cite[Lemma V\hspace{-.1em}I\hspace{-.1em}I.2.4]{HT01},  we have $\phi_{\pi'}=\bc_{E/F}(\phi_\pi)$.

\subsection{ An integral formula}
\label{subsec:integral_formula}

Assume that $n$ is even.
Let $\bfH$ be the centralizer of $\diag(1,  -1,  \ldots,  1,  -1)\in\bfG$.
We take an explicit isomorphism $\bfH\cong\Res_{E/F}(\GL_{n/2}\times\GL_{n/2})$ as follows.
Let $\sigma\in \bfG$ be the permutation matrix 
    \[
    \sigma=
        \begin{pmatrix}
        1 & 2 & \cdots & n/2 & n/2+1 & n/2+2 & \cdots & n \\
        1 & 3 & \cdots & n-1 & 2 & 4 & \cdots & n
        \end{pmatrix}.
    \]
Since $\diag(1,  -1,  \ldots,  1,  -1)=\sigma\diag(1_{n/2},  -1_{n/2})\sigma^{-1}$,  we have an isomorphism 
    \[
    \Res_{E/F}(\GL_{n/2}\times\GL_{n/2})\xrightarrow{\sim} \bfH;  \ (x_1,  x_2) \mapsto \sigma\diag(x_1,  x_2)\sigma^{-1}.
    \]
We set $\mathbf{h}(x_1,  x_2)=\sigma\diag(x_1,  x_2)\sigma^{-1}$.

Let $\bfH'$ be the subgroup of $\bfH$ given by
    \[
    \bfH'=\{\mathbf{h}(A,  A) \mid A\in\Res_{E/F}\GL_{n/2}\}.
    \]
Then $\bfH'\cong\Res_{E/F}\GL_{n/2}$.
Note that $\bfH$ and $\bfH'$ are $\tau$-stable and hence we can define the non-connected algebraic groups $\bfH^+=\bfH\rtimes\{1,  \tau\}$ and $\bfH'^+=\bfH'\rtimes\{1,  \tau\}$.
Set $\tilde{\bfH}=\bfH\tau$ and $\tilde{\bfH'}=\bfH'\tau$.

Recall that $\chi$ is a character of $E^\times$.
We define the character $\chi_\bfH$ of $\bfH(F)$ by 
    \[
    \chi_\bfH(\mathbf{h}(x_1,  x_2))=\chi(\det(x_1x_2^c)),  \qquad x_1,  x_2\in\GL_n(E).
    \]
Note that $\chi_\bfH(\tau(h))=\chi_\bfH(h)$ for all $h\in \bfH(F)$.
Hence we can extend $\chi_\bfH$ to a character of $\bfH^+$,  which is also denoted by $\chi_\bfH$,  by letting $\chi_\bfH(\tau)=1$.

We fix a set $\cT_\elliptic(\tilde{\bfH})$ (resp.\,$\cT_\elliptic(\tilde{\bfH}')$) of representatives of $\bfH(F)$-cnjugacy classes ($\bfH'(F)$-conjugcay classes) of $\tau$-stable maximal elliptic tori in $\bfH$ (resp.\,$\bfH'$).

For $\bfT\in\cT_\elliptic(\tilde{\bfH})$,  set $\tilde{\bfT}=\bfT\tau$ and let $\tilde{\bfT}(F)_{/\tau}$ be the set of $\bfT(F)$ be the set of $\bfT(F)$-conjugacy classes in $\tilde{\bfT}(F)$.
We set $W(\bfH,  \tilde{\bfT})=N_{\bfH(F)}(\tilde{\bfT})/\bfT(F)$,  where $N_{\bfH(F)}(\tilde{\bfT})$ is the normalizer of $\tilde{\bfT}$ in $\bfH(F)$.
Similarly for $\bfT'\in\cT_\elliptic(\tilde{\bfH}')$,  set $\tilde{\bfT}'=\bfT'\tau$ and let $\tilde{\bfT}'(F)_{/\tau}$ be the set of $\bfT'(F)$-conjugacy classes in $\tilde{\bfT}'(F)$.
We set $W(\bfH',  \tilde{\bfT}')=N_{\bfH'(F)}(\tilde{\bfT}')/\bfT'(F)$.

Take $\delta\in E\setminus F$ so that $\delta^2\in F$.
Let $H_0$ be the subgroup of $\GL_n(F)$ which consists of invertible matrices of the form
    \[
        h(A, B):=\begin{pmatrix}
        A & \delta^2 B \\
        B & A
        \end{pmatrix},  \quad A,  B\in \M_{n/2}(F).
    \]
Note that $H_0(F)\cong\GL_{n/2}(E)$.
Let $H'_0$ be the subgroup of $H_0$ consisting of $h(A,  \bzero_{n/2})$ with $A\in\GL_{n/2}(F)$.
Then $H'_0(F)\cong\GL_{n/2}(F)$.
We write the set of elliptic conjugacy classes in $(H_0/Z_{\GL_n})(F)$ (resp.\,$(H'_0/Z_{\GL_n})(F)$) by $\Gamma_\elliptic(H_0/Z_{\GL_n})$ (resp.\,$\Gamma_\elliptic(H'_0/Z_{\GL_n})$).

For $\bfT\in\cT_\elliptic(\tilde{\bfH})$,  there is an extension $L/E$ of degree $n/2$ such that $\bfT(F)\cong L^\times\times L^\times \cong (L\otimes_FE)^\times$.
Let $T\in\cT_\elliptic(H_0)$ be such that $T(F)\cong L^\times$.
The map $\tilde{g}\mapsto\cN(\tilde{g})$ induces a bijection from $\tilde{\bfT}_\elliptic(F)/W(\bfH,  \tilde{\bfT})$ to $T_\elliptic(F)/W(H_0,  T)$,  where $\tilde{\bfT}_\elliptic(F)=\bfT(F)\cap\tilde{\bfH}_\elliptic(F)$ and $T_\elliptic(F)=T(F)\cap H_{0,  \elliptic}(F)$.
Let $\Gamma_\elliptic(\tilde{\bfH}/Z_\bfG)$ be the set of elliptic conjugacy classes in $(\tilde{\bfH}/Z_\bfG)(F)$.
Then we have a bijection between $\Gamma_\elliptic(\tilde{\bfH}/Z_\bfG)$ and $\Gamma_\elliptic(H_0/Z_{\GL_n})$.
We write this correspondence as $\tilde{t}\leftrightarrow t$. 

For $\bfT'\in\cT_\elliptic(\tilde{\bfH}')$,  there is an extension $K'/E$ of degree $n/2$ such that $\bfT'(F)\cong K'^\times$.
Since $\bfT'$ is $\tau$-stable,  $K'$ is stable under $c\in\Gal(E/F)$.
Let $K\subset K'$ be the fixed field of $c$,  which is a degree $n/2$ extension of $F$ and $K'\cong K\otimes_FE$.
Let $T'\in\cT_\elliptic(H'_0)$ be such that $T'(F)\cong K^\times$.
Then the map $\tilde{g}\mapsto\cN(\tilde{g})$ induces a bijection from $\tilde{\bfT}'_\elliptic(F)/W(\bfH',  \tilde{\bfT'})$ to $T'_\elliptic(F)/W(H'_0,  T')$,  where $\tilde{\bfT}'_\elliptic(F)=\bfT'(F)\cap\tilde{\bfH}'_\elliptic(F)$ and $T'_\elliptic(F)=T'(F)\cap H'_{0,  \elliptic}(F)$.
Let $\Gamma_\elliptic(\tilde{\bfH}'/Z_\bfG)$ be the set of elliptic conjugacy classes in $(\tilde{\bfH}'/Z_\bfG)(F)$.
Then we have a bijection between $\Gamma_\elliptic(\tilde{\bfH}'/Z_\bfG)$ and $\Gamma_\elliptic(H'_0/Z_{\GL_n})$.
We write this correspondence as $\tilde{t}'\leftrightarrow t'$. 

For $\tilde{t}'\in\tilde{\bfT}'(F)\cap\tilde{\bfG}_\reg(F)$,  let $\bfG_{\tilde{t}'}$ (resp.\,$\bfH_{\tilde{t}'}$) be the centralizer of $\tilde{t}'$ in $\bfG$ (resp.\,$\bfH$).
Then $\bfG_{\tilde{t}'}(F)\cong\GL_2(K')$ and $\bfH_{\tilde{t}'}(F)\cong K'^\times\times K'^\times$.
Let $\fg_{\tilde{t}'}$ be the Lie algebra of $\bfG_{\tilde{t}'}$ and $\cO_{\tilde{t}'}\in\Nil(\fg_{\tilde{t}'}(F))$ the unique regular nilpotent orbit in $\fg_{\tilde{t}'}(F)$.

Set $\omega=(\chi\circ\Nm_{E/F})^{\frac{n}{2}}$ and regard it as a character of $Z_\bfG(F)\cong E^\times$.

\begin{definition}
For a quasi-character $\Theta$ on $\tilde{\bfG}(F)$ with central character $\omega$,  we set
    \begin{align*}
    \varepsilon_\geom(\Theta,  \chi)&=2\sum_{\tilde{\bfT}\in\cT_\elliptic(\tilde{\bfH})} |W(\bfH , \tilde{\bfT})|^{-1} 
    \int_{Z_\bfG(F)\bs \tilde{\bfT}(F)_{/\tau}}D^{\tilde{\bfH}}(\tilde{t}) \Theta(\tilde{t}) \rd \tilde{t} \\
    & \qquad +\sum_{\bfT'\in\cT_\elliptic(\tilde{\bfH}')} |W(\bfH',  \tilde{\bfT}')|^{-1} 
    \int_{Z_\bfG(F)\bs\tilde{\bfT}'(F)_{/\tau}}D^{\tilde{\bfH}'}(\tilde{t}')^2 
    c_{\Theta , \cO_{\tilde{t}'}}(\tilde{t}') \rd \tilde{t}'
    \end{align*}
For $\pi'\in\Irr_\disc(\bfG(F))$ with $\omega_{\pi'}=\omega$ such that $\pi'\cong\pi'^c$,  we set
    \[
    \varepsilon_\geom(\pi',  \chi)=\varepsilon_\geom(\Theta_{\tilde{\pi}'},  \chi),
    \]
where $\Theta_{\tilde{\pi}'}$ is the twisted character of $\tilde{\pi}'$.
\end{definition}

\begin{conjecture}\label{conj:epsilon}
Suppose $\pi'\in\Irr_\disc(\bfG(F))$ satisfies $\omega_{\pi'}=\omega$ and $\pi'\cong\pi'^c$ and $\rec_{\bfG(F)}(\pi')$ takes values in $\GSp_n(\C)$ with similitude factor $\chi\circ\Nm_{E/F}$.
Then we have
    \[
    \varepsilon_\geom(\pi',  \chi)= \varepsilon(\pi',  \chi),
    \]
where we set $\varepsilon(\pi',  \chi)=\varepsilon(\frac12,  \pi'\otimes\chi^{-1},  \psi_E)\chi(-1)^{\frac{n}{2}}$.
\end{conjecture}

Note that we have $\varepsilon(\frac12,  \pi'\otimes\chi^{-1},  \psi_E)=\varepsilon(\rec_{\bfG(F)}(\pi')\otimes\chi^{-1})$ by \cite{Hen02}.

\begin{theorem}\label{thm:epsilon}
Take $\pi\in\Irr_\disc(G(F))$.
Let $\pi_0=\JL(\pi)\in\Irr_\disc(\GL_n(F))$ be the Jacquet-Langlands transfer of $\pi$ to $\GL_n(F)$ and $\pi'=\bc_{E/F}(\pi_0)\in\Irr(\bfG(F))$ the base change of $\pi_0$ to $\bfG(F)$.
Assume that 
\begin{itemize}
\item \cref{hyp:2},  \cref{conj:mult_formula} and \cref{conj:epsilon} hold.
\item $\pi'\in\Irr_\disc(\bfG(F))$.
\end{itemize}
Then we have 
    \[
    m(\pi,  \chi)=\frac12(1+\chi_\pi(-1)\varepsilon(\pi',  \chi))
    \]
if $\rec_{\bfG(F)}(\pi')$ takes values in $\GSp_n(\C)$ with similitude factor $\chi\circ\Nm_{E/F}$ and  $m(\pi,  \chi)=0$ otherwise.
\end{theorem}

\begin{proof}
Recall that $\rec_{\bfG(F)}(\pi')$ takes values in $\GSp_n(\C)$ with similitude factor $\chi\circ\Nm_{E/F}$ if and only if $\rec_{G(F)}(\pi)$ takes values in $\GSp_n(\C)$ with similitude factor $\chi|_{F^\times}$.
We may assume this is the case since otherwise we have $m(\pi,  \chi)=0$ by \cref{prop:deduce}.

Take a central division algebra $D'$ over $F$ of dimension $\dim_F(D')=n^2$ and put $G_1=\GL_1(D')$.
Take $\pi_1\in\Irr_\disc(G_1(F))$ so that $\JL(\pi_1)=\pi_0$.
Note that $\pi_0,  \pi_1\in\Pi_\phi$,  $\chi_{\pi_0}(-1)=(-1)^n=1$ and $\chi_{\pi_1}(-1)=-1$.
From the same argument as the proof of \cref{prop:deduce} using the character relation \eqref{eq:DKV} we obtain
    \begin{align*}
    m_\geom(\pi_0,  \chi)-m_\geom(\pi_1,  \chi)&=
    2\int_{\Gamma_\elliptic(H_0/Z_{\GL_n})}D^{H_0}(t)\theta_{\pi_0}(t)\chi_{H_0}^{-1}(t) \rd t \\
    & \hspace{60pt}
    +\int_{\Gamma_\elliptic(H'_0/Z_{\GL_n})}D^{H'_0}(t')^2 c_{\pi,  \cO_{t'}}(t')\chi_{H_0}^{-1}(t') \rd t'.
    \end{align*}

The definition of $\varepsilon_\geom(\pi',  \chi)$ can be rewritten as 
    \begin{align*}
    \varepsilon_\geom(\pi',  \chi)&=
    2\int_{\Gamma_\elliptic(\tilde{\bfH}/Z_\bfG)}D^{\tilde{\bfH}}(\tilde{t})
    \Theta_{\tilde{\pi}'}(\tilde{t})\chi_{\bfH}^{-1}(\tilde{t}) \rd \tilde{t} \\
    & \hspace{60pt}
    +\int_{\Gamma_\elliptic(\tilde{\bfH}'/Z_\bfG)}D^{\tilde{\bfH}'}(\tilde{t}')^2 
    c_{\tilde{\pi}',  \cO_{\tilde{t}'}}(\tilde{t}')\chi_{\bfH}^{-1}(\tilde{t}') \rd \tilde{t}'.
    \end{align*}
For $t\in\Gamma_\elliptic(H_0/Z_{\GL_n})$ and $\tilde{t}\in\Gamma_\elliptic(\tilde{\bfH}/Z_\bfG)$ such that $\tilde{t}\leftrightarrow t$,  we have $D^{\tilde{\bfH}}(\tilde{t})=D^{H_0}(t)$,  $\chi_{\bfH}(\tilde{t})=\chi_{H_0}(t)$ and $\Theta_{\tilde{\pi}'}(\tilde{t})=\theta_{\pi_0}(t)$ by \eqref{eq:AC}.
Similarly for $t'\in\Gamma_\elliptic(H'_0/Z_{\GL_n})$ and $\tilde{t}'\in\Gamma_\elliptic(\tilde{\bfH}'/Z_\bfG)$ such that $\tilde{t}'\leftrightarrow t'$,  we have$D^{\tilde{\bfH'}}(\tilde{t}')=D^{H'_0}(t')$,  $\chi_{\bfH}(\tilde{t}')=\chi_{H_0}(t')$ and $c_{\tilde{\pi}',  \cO_{\tilde{t}'}}(\tilde{t}')=c_{\pi_0,  \cO_{t'}}(t')$.

Hence we have $\varepsilon_\geom(\pi',  \chi)=m_\geom(\pi_0,  \chi)-m_\geom(\pi_1,  \chi)$.
On the other hand,  by \eqref{eq:deduce2} we have $m_\geom(\pi_0,  \chi)+m_\geom(\pi_1,  \chi)=1$.

If $\chi_\pi(-1)=1$,  we have $m_\geom(\pi,  \chi)=m_\geom(\pi_0,  \chi)$ due to \eqref{eq:DKV} and hence
    \[
    1+\chi_\pi(-1)\varepsilon_\geom(\pi',  \chi)=2m_\geom(\pi_0,  \chi)=2m_\geom(\pi,  \chi).
    \]
Similarly,  if $\chi_\pi(-1)=-1$,  we have $m_\geom(\pi,  \chi)=m_\geom(\pi_1,  \chi)$ and hence
    \[
    1+\chi_\pi(-1)\varepsilon_\geom(\pi',  \chi)=2m_\geom(\pi_1,  \chi)=2m_\geom(\pi,  \chi).
    \]
Combined with \cref{conj:mult_formula} and \cref{conj:epsilon},  we obtain 
    \[
    2m(\pi,  \chi)=1+\chi_\pi(-1)\varepsilon(\pi',  \chi).
    \]
\end{proof}

\section{Twisted local relative trace formula}
\label{sec:twistedLRTF}

\subsection{Split linear period}
\label{subsec:split_linear}

We say that a representation $\pi'$ of $\bfG(F)$ is \emph{$(\bfH,  \chi_\bfH)$-distinguished} or $\pi'$ has \emph{split linear periods} with respect to $\chi_\bfH$ if $\Hom_{\bfH(F)}(\pi',  \chi_\bfH)\neq0$.
For $\pi'\in\Irr(\bfG(F))$,  set $\tilde{m}(\pi',  \chi)=\dim_\C\Hom_{\bfH(F)}(\pi',  \chi_\bfH)$.
Recall that $\omega=(\chi\circ\Nm_{E/F})^{\frac{n}{2}}$,  which is regarded as a character of $Z_\bfG(F)\cong E^\times$.
Note that $\tilde{m}(\pi',  \chi)=0$ unless $\omega_{\pi'}=\omega$.

We assume the following two hypotheses for split linear periods.
One is the multiplicity one property.
\begin{hypothesis}\label{hyp:3}
For $\pi'\in\Irr(\bfG(F))$,  we have $\tilde{m}(\pi',  \chi)\leq 1$
\end{hypothesis}
It is known by Chen and Sun \cite[Theorem B]{CS20} that \cref{hyp:3} holds for all but finitely many $\chi$ and the case of the trivial character is established by Jacquet and Rallis \cite{JR96}.

The other one is the relation between split linear periods and Shalika periods.
\begin{hypothesis}\label{hyp:4}
For $\pi'\in\Irr_\disc(\bfG(F))$,  $\tilde{m}(\pi',  \chi)\neq0$ if and only if $\pi'$ has Shalika periods with respect to the character $\chi\circ\Nm_{E/F}$.
\end{hypothesis}
The ``if'' part follows from \cite[Proposition 3.1]{FJ93}.
When $\chi$ is the trivial character,  the other direction is also known.
See \cite[Proposition 3.4]{LM17} and \cite[Theorem 5.1]{Mat14}.

Let $\bfP_{\mathrm{mir}}$ be the mirabolic subgroup of $G$,  the group of matrices with last row equal to $(0,  \ldots,  0,  1)$.
Suppose that $\chi$ is a unitary character and $\pi'\in\Irr_\gen(\bfG(F))$ is a unitary representation.
Then for $W\in\cW(\pi',  \psi_E)$,  the integral
    \[
    \ell_\chi(W):=\int_{(\bfH\cap \bfN)(F)\bs (\bfH\cap \bfP_{\mathrm{mir}})(F)}
    W(p) \chi_\bfH^{-1}(p) \rd p
    \]
converges \cite[Proposition 3.2]{LM15}.
Moreover,  it defines a nontrivial element of $\Hom_{\bfH(F)}(\pi',  \chi_\bfH)$ if $\tilde{m}(\pi',  \chi)\neq0$.
Recall that we set $\varepsilon(\pi',  \chi)=\varepsilon(\frac12,  \pi'\otimes\chi^{-1},  \psi_E)\chi(-1)^\frac{n}{2}$.

Let $\pi'\in\Irr_\disc(\bfG(F))$ such that $\tilde{m}(\pi',  \chi)\neq0$.
Assuming \cref{hyp:4},  $\pi'$ has Shalika model.
Hence by the functional equation \cite[Proposition 3.3]{FJ93},  we have 
    \[
    \ell_\chi(\tilde{\pi}'(\tau)W)=\varepsilon(\pi',  \chi)\ell_\chi(W),  \qquad
    W\in\cW(\pi',  \psi_E).
    \]
We define a bilinear form $\cL_{\pi',  \chi}\,\colon \pi'\times\pi'^\vee\to\C$ by 
    \[
    \cL_{\pi',  \chi}(v,  v^\vee)
    =\int_{Z_\bfG(F)\bs \bfH(F)}\langle \pi'(h)v,  v^\vee \rangle
    \chi_\bfH^{-1}(h) \rd h,  \qquad 
    v\in\pi',  \ v^\vee\in\pi'^\vee,
    \]
where $\langle \cdot,  \cdot\rangle$ is the canonical pairing on $\pi'\times\pi'^\vee$.
This integral converges absolutely since the symmetric pair $(\bfG,  \bfH)$ is tempered by \cite[Corollary 5.16]{GO16}.
If $\pi'$ is $(\bfH,  \chi_\bfH)$-distinguished and $\pi'^\vee$ is $(\bfH,  \chi_\bfH^{-1})$-distinguished,  then $\cL_{\pi',  \chi}$ defines a non-trivial element of $\Hom_{\bfH(F)\times \bfH(F)}(\pi'\boxtimes\pi'^\vee,  \chi_\bfH\boxtimes\chi_\bfH^{-1})$.
See \cite[Theorem 1.4]{CZhang16} and \cite[Proposition 4.2.1]{BP18}.

If we assume \cref{hyp:3},  there is a constant $c\in\C^\times$ such that
    \[
    \cL_{\pi',  \chi}(v,  v^\vee)=c\cdot \ell_\chi(W_v)\ell_{\chi^{-1}}(W_{v^\vee})
    \]
for all $v\in\pi'$,  $v^\vee\in\pi'^\vee$.
In particular,  we have
    \begin{equation}\label{eq:const}
    \cL_{\pi',  \chi}(\tilde{\pi}'(\tau)v,  v^\vee)=\varepsilon(\pi',  \chi)\cL_{\pi',  \chi}(v,  v^\vee).
    \end{equation}

\subsection{Geometric and spectral expansion}
\label{subsec:expansion2}

For a strongly cuspidal Schwartz-Harish-Chandra function $\tilde{f}\in\cC(\tilde{\bfG}(F),  \omega)$,  we define the twisted kernel function $\tilde{K}_{\tilde{f}}$ by
    \[
    \tilde{K}_{\tilde{f}}(g_1,  g_2)=\int_{Z_\bfG(F)\bs\tilde{\bfH}(F)} \tilde{f}(g_1^{-1}\tilde{h}g_2)
    \chi_\bfH^{-1}(\tilde{h}) \rd\tilde{h}.
    \]
The distribution $J(\tilde{f})$ is defined as
    \[
    J(\tilde{f})=\int_{\bfH(F)\bs\bfG(F)}\tilde{K}_{\tilde{f}}(g,  g)\rd g.
    \]

Define the geometric side of $J(\tilde{f})$ by $J_\geom(\tilde{f})=\varepsilon_\geom(\Theta_{\tilde{f}},  \chi)$.

\begin{conjecture}\label{conj:geom_exp2}
For a strongly cuspidal function $\tilde{f}\in\cC(\tilde{G},  \omega)$,  we have
    \[
    J(\tilde{f})=J_\geom(\tilde{f}),
    \]
where $\Theta_{\tilde{f}}$ denotes the (non-invariant) weighted orbital integral of $\tilde{f}$.
\end{conjecture}

Let $\tilde{f}\in\cC(\tilde{\bfG}(F),  \omega)$.
We define a Schwartz-Harish-Chandra function on $\bfG(F)$ by $f(g)=\tilde{f}(g\tau)$.
Note that if $\tilde{f}$ is cuspidal,  then so is $f$.
We have 
    \[
    \tilde{K}_{\tilde{f}}(g,  g)=\int_{Z_\bfG(F)\bs \bfH(F)} f(g^{-1}h\tau(g))\chi_\bfH^{-1}(h) \rd h.
    \]
    
By the Plancherel formula,  
    \[
    f(g)=\sum_{\substack{\pi'\in\Irr_\disc(\bfG(F)) \\ \omega_{\pi'}=\omega}}
    d(\pi')f_{\pi'}(g),  \qquad g\in \bfG(F),
    \]
where $d(\pi')$ is the formal degree of $\pi'$ and $f_{\pi'}(g)=\tr(\pi'^\vee(g^{-1})\pi'^\vee(f))$ is a linear combination of matrix coefficients of $\pi'$.

\begin{lemma}
Assume that \cref{hyp:3} and  \cref{hyp:4} hold.
For $\pi'\in\Irr_\disc(\bfG(F))$ such that $\omega_{\pi'}=\omega$,  we have
    \[
    \int_{\bfH(F)\bs\bfG(F)}\int_{Z_\bfG(F)\bs\bfH(F)} f_{\pi'}(g^{-1}h\tau(g))\chi_\bfH^{-1}(h) \rd h \rd g=
        \begin{cases}
        \varepsilon(\pi',  \chi)d(\pi')^{-1}\tilde{m}(\pi',  \chi)\Theta_{\tilde{\pi}'^\vee}(\tilde{f}_{\pi'})
        & \text{if $\pi'^c\cong\pi'$} \\
        0 & \text{otherwise}.
        \end{cases}
    \]
Here,  $\tilde{f}_{\pi'}\in\cC(\tilde{\bfG}(F),  \omega)$ is defined by $\tilde{f}_{\pi'}(\tilde{g})=f_{\pi'}(\tilde{g}\tau)$,  $\tilde{g}\in\tilde{\bfG}(F)$,  $d(\pi')$ is the formal degree of $\pi'$ and $\Theta_{\tilde{\pi}'^\vee}$ is the twisted character of $\tilde{\pi}'^\vee$.
\end{lemma}

\begin{proof}
Since $f_{\pi'}$ is a linear combination of matrix coefficients of $\pi'$,  we may suppose $f_{\pi'}(g)=f_{v,  v^\vee}(g):=\langle \pi'(g)v,  v^\vee \rangle$ for some $v\in\pi'$ and $v^\vee\in\pi'^\vee$.
It suffices to show 
    \begin{align*}
    &\int_{\bfH(F)\bs \bfG(F)}\int_{Z_\bfG(F)\bs \bfH(F)} f_{v,  v^\vee}(g^{-1}h\tau(g))\chi_\bfH^{-1}(h) \rd h \rd g \\
    &\hspace{100pt}=
        \begin{cases}
        \varepsilon(\pi',  \chi)d(\pi)^{-1}m(\pi',  \chi)\langle \tilde{\pi}'(\tau)v,  v^\vee \rangle
        & \text{if $\pi'^c\cong\pi'$} \\
        0 & \text{otherwise}.
        \end{cases}
    \end{align*}

Since 
    \begin{align*}
    \int_{Z_\bfG(F)\bs \bfH(F)} f_{v,  v^\vee}(g^{-1}h\tau(g))\chi_\bfH^{-1}(h) \rd h 
    &=\int_{Z_\bfG(F)\bs \bfH(F)} \langle \pi'(h\tau(g))v,  \pi'^\vee(g)v^\vee \rangle\chi_\bfH^{-1}(h) \rd h \\
    &=\cL_{\pi',  \chi}(\pi(\tau(g))v,  \pi^\vee(g)v^\vee),  
    \end{align*}
this integral is zero unless $\tilde{m}(\pi',  \chi)\neq0$.
 
Suppose that $\tilde{m}(\pi',  \chi)\neq0$.   
Take $v_0\in\pi'$ and $v_0^\vee\in\pi'^\vee$ so that $\cL_{\pi',  \chi}(v_0,  v_0^\vee)=1$.
Note that their image in the $(\bfH(F),  \chi_\bfH)$- and $(\bfH(F),  \chi_\bfH^{-1})$-coinvariants of $\pi'$ and $\pi'^\vee$ respectively,   form dual bases.
Thus the last expression becomes $\cL_{\pi',  \chi}(\pi'(\tau(g))v,  v_0^\vee)\cL_{\pi,  \chi}(v_0,  \pi'^\vee(g)v^\vee)$.

Hence 
    \begin{align*}
    &\int_{\bfH(F)\bs \bfG(F)}\int_{Z_\bfG(F)\bs \bfH(F)} f_{v,  v^\vee}(g^{-1}h\tau(g))\chi_\bfH^{-1}(h) \rd h \rd g \\
    &\hspace{50pt}=\int_{\bfH(F)\bs\bfG(F)} 
    \cL_{\pi',  \chi}(\pi'(\tau(g))v,  v_0^\vee)\cL_{\pi',  \chi}(v_0,  \pi'^\vee(g)v^\vee) \rd g \\
    &\hspace{50pt}=\int_{Z_\bfG(F)\bs \bfG(F)} \cL_{\pi',  \chi}(\pi'(\tau(g))v,  v_0^\vee) 
    \langle v_0,  \pi'^\vee(g)v^\vee \rangle \rd g \\
    &\hspace{50pt}=\int_{Z_\bfG(F)\bs \bfH(F)}\chi_\bfH^{-1}(h) \int_{Z_\bfG(F)\bs \bfG(F)} 
    \langle \pi'(\tau(g))v,  \pi'^\vee(h^{-1})v_0^\vee \rangle \langle v_0,  \pi'^\vee(g)v^\vee \rangle \rd g \rd h
    \end{align*}
By the Schur orthogonality relations,  the last integral vanishes unless $\pi'\cong\pi'\circ\tau$,  or equivalently $\pi'\cong\pi'^c$.

Suppose that $\pi'\cong\pi'^c$.
Then  from \eqref{eq:const},  
    \[
    \cL_{\pi',  \chi}(\pi'(\tau(g))v,  v_0^\vee)=\cL_{\pi',  \chi}(\tilde{\pi}'(\tau)\pi'(g)\tilde{\pi}'(\tau)v,  v_0^\vee) 
    =\varepsilon(\pi',  \chi) \cL_{\pi',  \chi}(\pi'(g)\tilde{\pi}'(\tau)v,  v_0^\vee).
    \]
Therefore we get
\begin{align*}
    &\int_{\bfH(F)\bs\bfG(G)}\int_{Z_\bfG(F)\bs\bfH(F)} f_{v,  v^\vee}(g^{-1}h\tau(g))\chi_\bfH^{-1}(h) \rd h \rd g \\
    &\hspace{40pt}
    =\varepsilon(\pi',  \chi) \int_{Z_\bfG(F)\bs\bfG(F)} \cL_{\pi,  \chi}(\pi'(g)\tilde{\pi}'(\tau)v,  v_0^\vee) 
    \langle v_0,  \pi'^\vee(g)v^\vee \rangle \rd g \\
    &\hspace{40pt}
    =\varepsilon(\pi',  \chi)d(\pi')^{-1}\cL_{\pi',  \chi}(v_0,  v_0^\vee)\langle \tilde{\pi}'(\tau)v,  v^\vee \rangle
     =\varepsilon(\pi',  \chi)d(\pi)^{-1}\tilde{m}(\pi',  \chi)\langle \tilde{\pi}'(\tau)v,  v^\vee \rangle.
    \end{align*}
\end{proof}

Define the distribution $J_\spec(\tilde{f})$,  the spectral side of $J(\tilde{f})$ by
    \[
    J_\spec(\tilde{f})=\sum_{\substack{\pi'\in\Irr_\disc(\bfG(F)) \\ 
    \omega_\pi'=\omega, \ \pi'\cong\pi'^c}}
     \varepsilon(\pi',  \chi)\tilde{m}(\pi',  \chi)\Theta_{\tilde{\pi}'^\vee}(\tilde{f}).
    \]
We obtain the following spectral expansion of $J(\tilde{f})$.

\begin{corollary}\label{cor:spec}
Assume that \cref{hyp:3} and  \cref{hyp:4} hold.
For a cuspidal function $\tilde{f}\in\cC(\tilde{\bfG},  \omega)$,  we have $J(\tilde{f})=J_\spec(\tilde{f})$.
\end{corollary}

\begin{proposition}
Assume that \cref{hyp:2},  \cref{hyp:3} and  \cref{hyp:4} hold.
Then \cref{conj:epsilon} follows from \cref{conj:geom_exp2}.
\end{proposition}

\begin{proof}
Let $\pi\in\Irr_\disc(\bfG(F))$ such that $\omega_{\pi'}=\omega$ and $\pi'\cong\pi'^c$.
Take a matrix coefficient $\tilde{f}\in\cC(\tilde{\bfG},  \omega)$ of $\tilde{\pi}'$ satisfying $\tilde{f}(\tau)\neq0$.
Then we have
    \[
    \Theta_{\tilde{f}}(\tilde{x})=d(\pi')^{-1}\tilde{f}(\tau)\Theta_{\tilde{\pi}'}(\tilde{x}),  
    \qquad \tilde{x}\in\tilde{\bfG}_\reg(F).
    \]
From this,  we also obtain
    \[
    c_{\tilde{f},  \cO_{\tilde{t}'}}(\tilde{t}')
    =d(\pi')^{-1}\tilde{f}(\tau)c_{\tilde{\pi}',  \cO_{\tilde{t}'}}(\tilde{t}')
    \]
for all $\bfT'\in\cT_\elliptic(\bfH')$ and $t'\in \tilde{\bfT}'(F)\cap\tilde{\bfG}_\reg(F)$.
Hence we have $J_\geom(\tilde{f})=d(\pi')^{-1}\tilde{f}(\tau)\varepsilon_\geom(\pi',  \chi)$.

On the other hand,  by the Schur orthogonality relations,  we have
    \[
    J_\spec(\tilde{f})= \varepsilon(\pi',  \chi)\tilde{m}(\pi',  \chi)\Theta_{\tilde{\pi}'^\vee}(\tilde{f})
    = \varepsilon(\pi',  \chi)d(\pi')^{-1}\tilde{f}(\tau)\tilde{m}(\pi',  \chi).
    \]
Hence we obtain 
    \[
    \varepsilon_\geom(\pi',  \chi)=\varepsilon(\pi',  \chi)\tilde{m}(\pi',  \chi).
    \]

Since we assume \cref{hyp:3} and \cref{hyp:4},  $\tilde{m}(\pi',  \chi)=1$ is equivalent to that $\pi'$ has Shalika periods with respect to $\chi\circ\Nm_{E/F}$.
By \cref{hyp:2},  this is equivalent to that $\rec_{\bfG(F)}(\pi')$ takes values in $\GSp_n(\C)$ with similitude factor $\chi\circ\Nm_{E/F}$.
This proves \cref{conj:epsilon}.
\end{proof}

\begin{Ac}
The author would like to thank Chen Wan for answering many questions,  Yugo Takanashi for helpful discussions and the anonymous referee for careful reading and a lot of helpful comments.
This research is partially supported by JSPS KAKENHI (JP22K13891 and JP23K20785).
\end{Ac}

\end{document}